\documentclass[reqno]{amsart}
\usepackage{amsthm,amssymb,amsmath,epsfig,graphics,appendix,bm, mathrsfs,bbm}
\usepackage{graphicx}
\usepackage{hyperref}
\usepackage{bm}
\usepackage[dvipsnames, svgnames, x11names]{xcolor}
\usepackage{stmaryrd} 
\usepackage[left=1in, right=1in, top=1.1in,bottom=1.1in]{geometry}
\usepackage{enumitem} 

\def\R{\mathbb R}
\def\E{\mathbb E}

\newtheorem{thm}{Theorem}[section]
\newtheorem{rem}{Remark}[section]

\DeclareMathOperator*{\esssup}{ess\,sup}

\usepackage{geometry}
\usepackage{xcolor}
\usepackage{amsfonts}
\usepackage{lipsum}                     
\usepackage{xargs}         
\usepackage[colorinlistoftodos,prependcaption,textsize=tiny, textwidth=2cm]{todonotes} 

\usepackage{framed}
\usepackage{csquotes}
 
\usepackage{mathrsfs}
\usepackage{bm}
\usepackage{mathtools}

\usepackage{cleveref}
\crefname{assumptionalt}{Assumption}{Assumptions}

\numberwithin{equation}{section}

\newcommand{\ome}{\omega}

\newcommand{\MV}{\mathcal{V}}
\newcommand{\MF}{\mathcal{F}}
\newcommand{\MB}{\mathcal{B}}
\newcommand{\MU}{\mathcal{U}}
\newcommand{\MA}{\mathcal{A}}

\newcommand{\MX}{\mathcal{X}}

\newcommand{\tx}{\tilde{x}}

\newcommand{\tB}{\tilde{B}}

\newcommand{\tX}{\tilde{X}}
\newcommand{\tY}{\tilde{Y}}

\newcommand{\ty}{\tilde{y}}
\newcommand{\tZ}{\tilde{Z}}

\newcommand{\tb}{\tilde{b}}

\newcommand{\tA}{\tilde{A}}

\newcommand{\tf}{\widetilde{f}}

\newcommand{\tF}{\tilde{F}}
\newcommand{\tJ}{\tilde{J}}

\newcommand{\tM}{\tilde{M}}
\newcommand{\tN}{\tilde{N}}
\newcommand{\tD}{\tilde{D}}
\newcommand{\tC}{\tilde{C}}
\newcommand{\tG}{\tilde{G}}
\newcommand{\tih}{\tilde{h}} 
\newcommand{\tg}{\tilde{g}} 
\newcommand{\tz}{\tilde{z}} 
\newcommand{\tMV}{\tilde{\MV}}
\newcommand{\tH}{\tilde{H}} 
\newcommand{\tsig}{\tilde{\sigma}}

\newcommand{\hu}{\hat{u}} 
\newcommand{\hX}{\hat{X}} 
\newcommand{\hY}{\hat{Y}} 
\newcommand{\hZ}{\hat{Z}} 
\newcommand{\hB}{\hat{B}} 
\newcommand{\hD}{\hat{D}} 
\newcommand{\hM}{\hat{M}} 
\newcommand{\hG}{\hat{G}}

\newcommand{\bY}{\bar{Y}}
\newcommand{\bX}{\bar{X}}

\newcommand{\bxi}{\bar{\xi}}
\newcommand{\bZ}{\bar{Z}}

\newcommand{\by}{\bar{y}}
\newcommand{\bz}{\bar{z}}
\newcommand{\bBBeta}{\bar{\bm{\eta}}}
\newcommand{\balpha}{\bar{\alpha}}

\newcommand{\bMA}{\bar{\MA}}
\newcommand{\bu}{\bar{u}}
\newcommand{\bJ}{\bar{J}}
\newcommand{\bMV}{\bar{\MV}}

\newcommand{\BD}{\mathbf{D}}
\newcommand{\EB}{\E_{\bullet}}

\newcommand{\N}{\mathbb N}

\renewcommand{\P}{\mathbb P}

\newcommand{\mc}{\mathscr{C}}
\newcommand{\md}{\mathbf{D}}
\newcommand{\mb}{\mathcal{B}}

\newcommand{\mf}{\mathcal{F}}
\newcommand{\ml}{\mathcal{L}}

\newcommand{\mpe}{\mathcal{P}}

\newcommand{\ma}{\mathcal{A}}
\newcommand{\mbb}{\mathcal{B}}
\newcommand{\U}{\mathcal{U}}

\newcommand{\D}{\mathscr{D}}

\newcommand{\BBeta}{\bm{\eta}}
\newcommand{\BX}{\bm{X}}

\newcommand{\be}{\begin{equation}}
\newcommand{\ee}{\end{equation}}
\newcommand{\BB}{\mathbf{B}}

\newcommand{\vertiii}[1]{{\left\vert\kern-0.25ex\left\vert\kern-0.25ex\left\vert #1 
    \right\vert\kern-0.25ex\right\vert\kern-0.25ex\right\vert}}

\DeclareMathAlphabet{\mathdutchcal}{U}{dutchcal}{m}{n}
\newcommand{\mn}{ \eta^{(2)}}
\newcommand{\mnb}{ \bar{\eta}^{(2)} }

\newcommand{\sroughnogeo}{{\mc^\alpha}}

\newcommand{\vep}{\varepsilon}

\newtheorem{prop}[thm]{Proposition}
\newtheorem{cor}[thm]{Corollary}
\newtheorem{lemma}[thm]{Lemma}
\newtheorem{definition}[thm]{Definition}

\newtheorem{example}[thm]{Example}
\newtheorem{Assumption}[thm]{Assumption}

\newcommand{\defeq}{\vcentcolon=}

\begin{document}

\title[rough stochastic PMP]{Pontryagin Maximum Principle for rough stochastic systems and pathwise stochastic control}\thanks{We thank Peter Karl Friz and Khoa L\^e for helpful remarks and fruitful discussions. Financial support
from the DFG CRC/TRR 388 “Rough Analysis, Stochastic Dynamics and Related Fields”, Projects B04
and B05, is gratefully acknowledged. Zhang is partially supported by the Fundamental Research Funds for
the Central Universities, NSF of China and Shandong (Grant Numbers 12031009, ZR2023MA026), Young
Research Project of Tai-Shan (No.tsqn202306054).}
 \author[U. Horst]{Ulrich Horst}\address{Departemnt of Mathematics and School of Business and Economics, Humboldt University Berlin, Unter den Linden 6, 10099 Berlin, Germany. }
 \email{horst@math.hu-berlin.de}

 \author[H. Zhang]{Huilin Zhang}\address{Research Center for Mathematics and Interdisciplinary Sciences, Frontiers Science Center for Nonlinear Expectations, Ministry of Education, Shandong University, Qingdao 266237, China; Department of Mathematics, Humboldt University Berlin, Unter den Linden 6, 10099 Berlin, Germany.} 
 \email{huilinzhang@sdu.edu.cn}




\begin{abstract}
We analyze a novel class of rough stochastic control problems that allows for a convenient approach to solving pathwise stochastic control problems with both non-anticipative and anticipative controls. We first establish the well-posedness of a class of controlled rough SDEs with affine rough driver and establish the continuity of the solution w.r.t.~the driving rough path. This allows us to define pathwise stochastic control problems with anticipative controls. Subsequently, we apply a flow transformation argument to establish a necessary and sufficient maximum principle to identify and characterize optimal strategies for rough and hence pathwise stochastic control problems. We show that the rough and the corresponding pathwise stochastic control problems share the same value function. For the benchmark case of linear-quadratic problems with bounded controls a similar result is shown for optimal controls. 
\end{abstract}

\maketitle


\vspace{1mm}



\renewcommand{\baselinestretch}{1.0}
\setlength{\parskip}{4pt}

\mbox{ }

\section{Introduction}

A standard stochastic control problem comprises a stochastic differential equation (SDE) of the form
 \begin{equation} \label{SDE} 
 \begin{split}
	dX_s & = b(s,X_s, u_s) ds + \sigma(s,X_s,u_s) d W_s,  \quad 0 \le t \le s \le T 
\end{split} 
\end{equation}
with an initial condition $X_t =x$, where $W$ is a multi-dimensional Brownian motion and $ u = (u_t)_{t \in [0, T]}$ is a progressively measurable control, and the objective is to minimize an expected cost functional of the form 
\begin{equation}\label{eq: cost_function_intro-}
    J(  u; t, x) = \E \left[ \int_t^T h(s, X_s, u_s)ds + g(X_T )  \Big{|} \mf^W_t \right]
\end{equation}
over a set of admissible controls. 
Textbook control problems of the above form have been extended in numerous directions, including mean-field and singular control problems, which are by now also well understood. 

A less well-understood extension is pathwise stochastic control. Pathwise stochastic control models arise naturally when a decision maker can condition her choice on the realization of some exogenous noise process. Examples include stochastic volatility and utility optimization models where investors can condition their decisions on additional information such as the realization of some volatility process \cite{BBFP, buckdahn2007pathwise} as well as deep learning networks \cite{bayer2023stability, gassiat2024gradient} in infinite layer limits, where the future of processes resulting from the initialization of deep neural network weights is fully available.     

The difference between classical control problems including models with partial observations and pathwise stochastic control problems boils down to measurability assumptions on admissible controls. Depending on the specific application, different measurability conditions may arise that require different mathematical settings, depending on whether controls may or may not - {\sl explicitly} - depend on the future of a given noise process.   


\subsection{Pathwise stochastic control problems}\label{Sec:introtopathwise}

The interest in pathwise stochastic and anticipative control goes back at least to Davis and Burstein \cite{DB92} who considered a stochastic optimal control problem with anticipative controls as a family of deterministic control problems parametrized by the paths of the driving Wiener process and a Lagrange multiplier that is adapted to future increments of the Wiener process. Similar Lagrange multipliers were used to a model of optimal stopping in the follow-up work \cite{DK94}. 

Lions and Souganidis \cite{LS98, LS98b} introduced a pathwise stochastic control problem, where one controls the solution of the equation \eqref{SDE} for a given realization of some component of the driving Brownian motion. Specifically, they considered a stochastic control problem where the state dynamics follows an SDE of the form 
\begin{equation}\label{eq:double_bm_dynamics_intro2}
\begin{split}
    dX^{{  u}}_s & = b(s, X^{{  u}}_s, u_s)ds + \sigma(s, X^{{  u}}_s, u_s)dW_s + \hat \sigma(s, X^{{  u}}_s) \circ dB_s(\omega), \quad s \in [t,T] \\
    X_t & = x
\end{split}
\end{equation}
where $W$ and $B$ are independent Brownian motions whose canonical filtrations are denoted $\mf^W$ and $\mf^B$, respectively, $B_\cdot(\omega)$ is a given realization/path of the Brownian motion $B$, and $\circ dB$ denotes the Stratonovich integral. Their objective is to minimize a conditional expected cost functional of the form
\begin{equation}\label{eq:conditional_cost_function_intro'}
    J( u; t, x) = \E \left[ \int_t^T h(s, X^{{  u}}_s, u_s)ds + g(X^{{  u}}_T ) \Big{|} \mf^W_t \otimes \mf^B_T \right]
\end{equation}
over a set $\MA$ of $({\mf}^W_t)_{t \in [0,T]}$-progressively measurable (``non-anticipative'') controls ${  u}$. Buckdahn and Ma \cite{buckdahn2007pathwise} proved a dynamic programming principle for this problem that allowed them to prove that the value function  
\begin{equation} \label{SDE12} 
	v(\omega, t, x) := \mbox{essinf}_{{ u} \in \ma} J(  u; t, x)
\end{equation}
is a stochastic viscosity solution to a non-linear stochastic HJB equation. To overcome the many and often subtle challenges that arise from the fact that the value function is now a random field for which no minimizing sequence may exist, they employed a Doss–Sussmann transformation to reduce their pathwise stochastic control problem to a more standard control problem that they called the {\sl wider sense control problem}. 

The wider sense problem features a state dynamics that follows an SDE driven only by $W$, but whose coefficients depend on the Brownian path $B(\omega)$. Moreover, the set of admissible controls is given by a set of $\tilde \ma$ of $(\mf^W_t \otimes \mf^B_T)_{t \in [0, T]}$-progressively measurable processes that depend on the full path of $B$ (``anticipative controls''). Under suitable boundedness and continuity assumptions on the model parameters, it was then shown that the value functions of the pathwise stochastic and the wider sense control problem coincide.  

A priori, only nonanticipative controls are allowed in \cite{buckdahn2007pathwise,LS98, LS98b}. Anticipative controls were primarily a vehicle to solve the wider sense and eventually the pathwise control problem. At the same time, one would expect {\sl optimal} controls to be anticipative; a specific example where this is indeed the case is given in \cite{buckdahn2007pathwise}. This suggests to work right away with anticipative controls, and to consider instead the value function 
\begin{equation} \label{SDE13} 
	\tilde v(\omega, t, x) := \mbox{essinf}_{{  v} \in \tilde \ma} J(  v; t, x).
\end{equation}

This value function corresponds to an {\sl anticipative pathwise stochastic  control problem} where the state dynamics is perturbed by an observable, yet uncontrolled exogenous noise. A typical example would be the representative player's optimization problem in a mean-field game with deterministic common noise if one is interested in understanding how the mean-field game equilibrium varies with the exogenous perturbation. 

Working with anticipative controls is closer in spirit to traditional stochastic control, where it is commonly assumed that controls are adapted to the filtration with respect to which conditional expectations are computed. In the pathwise stochastic control setting, this is the filtration $(\mf^W_t \otimes \mf^B_T)_{t \in [0, T]}$.  

Using anticipative controls has been hindered so far by the fact that the state dynamics \eqref{eq:double_bm_dynamics_intro2} is ill-posed as a standard It\^o or Stratonovich integral equation when controls are allowed to depend on the full path $B(\omega)$. In particular, it may not be possible to define the optimal state dynamics of a non-anticipative pathwise stochastic control problem in the framework of traditional stochastic analysis. 

In this paper, we introduce a mathematical framework based on rough path theory that allows us to introduce a broad class of rough stochastic control problems, in terms of which we can define pathwise stochastic control problems with anticipative controls. We establish necessary and sufficient maximum principles for such problems and show that rough and anticipative control problems are equivalent. 


\subsection{Rough stochastic control problems}

To the best of our knowledge, Diehl et al. \cite{Diehl2013StochasticCW} were the first to apply rough path theory to optimal control. They considered a controlled dynamics of the form 
 \begin{equation} \label{SDE2} 
 \begin{split}
	dX_s & = b(X_s, u_s) ds + \Sigma(X_s) d {\bm \eta}_s,  \quad s \in [t,T] \\
	 X_t & = x
\end{split} 
\end{equation}
driven by a geometric rough path ${\bm \eta}$. They obtained a version of Pontryagin’s maximum principle, characterized their value function as the unique solution of a rough HJB equation, and obtained a duality result for the corresponding non-anticipative stochastic control problem. 

As in \cite{buckdahn2007pathwise,LS98, LS98b}  the diffusion coefficient in \eqref{SDE2} is uncontrolled. The existing literature on pathwise stochastic control does indeed mostly focus on the case where the control process appears in the drift, but not in the diffusion term. A notable exception is the work of Allan and Cohen \cite{allan2020pathwise}. They studied the degeneracy phenomenon induced by directly controlling the diffusion coefficient in great detail, and introduced a method to resolve the degeneracy whilst retaining dynamic programming, albeit at the cost of altering the primal problem in a significant manner.

We consider a class of rough stochastic control problems where, for any admissible control $ u \in \ma$ the state dynamics follows of rough stochastic differential equations (rSDEs) of the form 
\begin{equation}\label{eq:rough_stochastic_dynamics_2}
 \begin{split}
    dX^{\bm{\eta}, u}_s & = b(s, X^{\bm {\eta}, u }_s, u_s)ds + \sigma( s, X^{\bm {\eta}, u}_s,u_s)dW_s + \hat \sigma(  s, X^{\bm {\eta}, u}_s) d\bm{\eta}_s, \quad s \in [t,T] \\
   X^{\bm{\eta},u}_t & = x
\end{split}
\end{equation}
that feature both a Brownian and a rough path integral driven by a deterministic geometric rough path $\bm{\eta}$ with H\"older exponent $\alpha \in (\frac{1}{3}, \frac{1}{2}]$ (concrete assumptions of $(b,\sigma, \hat{\sigma})$ for the well-posedness of the above equation are given in Assumption \ref{ass:m-linear-growth_lip}). 

Building on L\^e's stochastic sewing lemma \cite{Le20} a meaning to {\sl uncontrolled} rSDEs with {\sl bounded} coefficients was first given in \cite{friz2021roughito}; see  \cite{FLZ25} for a review of related applications of rSDEs. Their result has recently been extended to a class of linear but deterministic (hence uncontrollable) coefficients in \cite{BCN24}. We establish the existence and uniqueness of solutions to (controlled) rSDEs with unbounded, random coefficients. Motivated by the analysis of linear-quadratic (LQ) problems our key assumption is that rough driver $\hat \sigma$ is affine in the state variable; all other coefficients may be non-linear and stochastic. 

The goal is then to minimize the cost functional 
\begin{equation}\label{eq: cost_function_intro}
    J^{\bm \eta}(  u; t, x) = \E \left[ \int_t^T h(s, X^{\bm {\eta}, u}_s, u_s)ds + g(X^{\bm {\eta}, u}_T )  \Big{|} \mf^W_t \right]
\end{equation}
over a set of non-anticipative controls $\ma$ (assumptions of $(h,g)$ are given in Assumption \ref{ass:control_assumptions}). The running and terminal cost function may again be unbounded. 
The corresponding value function is denoted
\[
	\bar v({\bm \eta},t,x) := \mbox{essinf}_{{  u} \in \ma} J^{\bm \eta}(  u; t, x).
\]

Having shown that the state dynamics is well defined, we apply the Doss-Sussmann-type transformation 
\begin{equation}\label{DS10}
	\phi_s(x) = x + \int_t^s \hat \sigma_r(\phi_r(x)) d\bm \eta_r
\end{equation}
to reduce the rough stochastic control problem to a standard control problem for which necessary and sufficient maximum principles are readily available. Undoing the transformation allows us to obtain the desired SMPs for rough control problems, the second main contribution of this paper. 

Our transformed control problem can be viewed as a version of the wider sense control problem in \cite{buckdahn2007pathwise}. 
Since our existence results for rough SDEs with affine driver are required to prove that the transformation \eqref{DS10} is well defined, we first (must) give meaning to the state dynamics before we can introduce our notion of the wider sense control problem. 


\subsection{From rough stochastic to anticipating control problems}

Having solved the rough control problem, which is interesting in its own right, we return to the original pathwise stochastic control problem. 

Since every Brownian path $B(\omega)$ can almost surely be lifted to a Brownian rough path ${\bm B}(\omega)$ (either in the It\^o sense or the Stratonovich sense), there are at least two ways to link pathwise stochastic and rough stochastic control problems. The first is to evaluate the value function of the rough problem at ${\bm \eta} = {\bm B}(\omega)$ i.e.~to consider the function 
\[
    \Big( \mbox{inf}_{{  u} \in \ma} J^{{\bm \eta}}(  u; t, x) \Big) \Big |_{{\bm \eta = \bm B(\omega)}}. 
\]

Our results on the rough stochastic control problem guarantee that this approach is well-defined, and our maximum principle for rough stochastic control problems provides necessary and sufficient conditions for controls to be optimal in the pathwise, non-anticipative setting.

The second approach is to substitute the Brownian rough path into the dynamics \eqref{eq:rough_stochastic_dynamics_2} and solve the corresponding control problem, which results in the value function
\[
    \mbox{essinf}_{{  u} \in \ma} J^{{\bm B}(\omega)}(  u; t, x) \qquad \mbox{respectively,} \qquad 
    \mbox{essinf}_{{  u} \in \bar \ma} J^{{\bm B}(\omega)}(  u; t, x)
\]
depending on whether only non-anticipative or anticipative controls are admissible. When only non-anticipative controls are allowed it has been shown in the companion paper \cite{FLZ24} that both approaches are equivalent under certain conditions, that is, a.s.
\[
    \Big( \mbox{inf}_{{  u} \in \ma} J^{{\bm \eta}}(  u; t, x) \Big) \Big |_{{\bm \eta = \bm B(\omega)}} =
    \mbox{essinf}_{{  u} \in \ma} J^{{\bm B}(\omega)}(  u; t, x).
\]
In that work, the authors also proved a dynamical programming principle and related rough HJB equations for controlled rough SDEs with {\sl bounded} coefficients. Our focus will be on unbounded coefficients. 

Our existence of solution results for rough SDEs allows us to give meaning to the state dynamics \eqref{eq:rough_stochastic_dynamics_2} and hence to justify the second approach and to define the corresponding control problem if anticipative controls are allowed. In fact, any anticipative control ${  v} \in \bar \ma$ is of the form
\[
	v_t(\omega) = u(t,\omega,{\bm B}(\omega))
\]
for some measurable function $u$ where $u(\cdot,\cdot,{\bm \eta})$ is a non-anticipative control, for all $\bm \eta$ from a suitable set. This allows us to define the solution $X^{{  v}}$ to equation \eqref{eq:rough_stochastic_dynamics_2} almost surely as
\[
	X^{{  v}}(\omega) := X^{{\bm B}(\omega),{  u}(\cdot,{\bm B}(\omega))}(\omega).
\]

Under suitable assumptions on the model parameters we prove that the random variable
\[
    \omega \mapsto X^{{\bm B}(\omega),{  u}(\cdot,{\bm B}(\omega))}(\omega) 
\]    
is measurable (this is not trivial) so that the conditional cost functional \eqref{eq:conditional_cost_function_intro'}, and hence the pathwise stochastic control problem, is well defined up to an integrability assumption on the cost coefficients. Moreover, we show that the SMP for rough stochastic control problems yields an SMP for pathwise stochastic control problems and that under an additional continuity assumption on the cost function, the randomized rough stochastic and the pathwise stochastic problem share the same value function: 
\[
    \Big( \mbox{inf}_{{  u} \in \ma} J^{{\bm \eta}}(  u; t, x) \Big) \Big |_{{\bm \eta = \bm B(\omega)}} =
    \mbox{essinf}_{{  u} \in \bar \ma} J^{{\bm B}(\omega)}(  u; t, x).
\]

This result is not obvious, as the sets of admissible controls in the two settings are very different. In the rough setting, admissible controls are required to be $\mf^W$-progressively measurable; no a priori measurability of controls w.r.t.~to the driving rough path is required. By contrast, anticipative controls are measurable w.r.t.~ Brownian paths.  

If the rough problem admits an optimal control $u^*(\cdot,\BBeta)$ that is measurable in the rough path $\BBeta$ as is the case for LQ problems, then $u^*(\cdot,{\bm B}(\omega))$ is almost surely an optimal control to the anticipative control problem, and the optimal state dynamics is well defined as the solution to a rSDE.

In conclusion, our contributions are twofold. First, on the rough stochastic analysis side, we establish the well-posedness of rough SDEs driven by affine rough stochastic drivers, and we derive a representation of the rough Doss–Sussmann transformation when the driver is an affine stochastically controlled function. Second, we apply this rough stochastic analysis framework to obtain Pontryagin’s maximum principle—both necessary and sufficient conditions—for pathwise stochastic control problems, allowing for admissible controls that may be anticipative. Our framework accommodates linear–quadratic models that were beyond the scope of earlier works.

The remainder of this paper is organized as follows. Section 2 reviews selected results from the rough path theory. Section 3 establishes the existence, uniqueness, and regularity of solutions for stochastic systems with affine rough drivers. Section 4 introduces the rough control problem and establishes maximum principles for such systems. Section 5 introduces a framework for solving anticipative pathwise control problems.  


\section{Preliminaries on rough paths and rough stochastic integrals}\label{sec:forward_rSDE}

We start recalling basic results on rough paths, rough stochastic integrals and related rough analysis results that will be used throughout. We refer to the textbook \cite{friz2021roughito} for more details on rough analysis. 

In what follows, $\alpha \in (\frac13, \frac12)$ is a H\"older coefficient, $[0,T]$ is a fixed time interval, and $(V,|\cdot|)$ is a Euclidean space. $\ml(V^1;V^2)$ denotes the space of linear functions between the Euclidean spaces $V^1$ to $V^2$. 

All random variables and stochastic processes are defined on a common complete filtered probability space $(\Omega,\mf,(\mf_t)_{t},\mathbb{P})$ that supports a $d_W$-dimensional Brownian motion $W$. For any $m < \infty$ we denote by $L^m$ the space of all $V$-valued random variables $\xi$ such that 
\[
 \lVert \xi \rVert_m \defeq \E[\lvert \xi \rvert^m]^{\frac{1}{m}} < \infty.
\]
The space of all a.s.~finite $V$-valued random variables is denoted $L^\infty$ and
\[
    \lVert \xi \rVert_\infty \defeq \esssup_\omega |\xi(\omega) |
\]

For $2\le m \le n$ we denote by $H^{m,n}({[s,t]};V)$ with $H^m({[s,t]};V) := H^{m,m}({[s,t]};V)$ and  $S^m({[s,t]};V)$, the spaces of adapted, $V$-valued progressively measurable processes $X$, equipped with the respective norms 
$$
\|X \|_{H^{m,n},{[s,t]}}:=\left[ \E  [\int_s^t \lvert X_r \rvert^n dr  ]^\frac{m}{n}\right]^{\frac{1}{m}} ,\ \ \ \| X \|_{S^m,[s,t]}:=  \lVert \sup_{r \in [s, t]} \lvert X_r \rvert \rVert_m < \infty.
$$

For any function $f:\R^d \rightarrow V$ that has continuous bounded derivatives up to order $n$, we write 
$$
|f|_n:=\max_{i=0,...,n} \sup_{x\in \R^d}|D^if(x)|. $$
For a one-parameter function/path $\eta:[0,T] \rightarrow V$ we set $\delta \eta_{s,t}:= \eta_t - \eta_s.$ The space of $V$-valued $\alpha$-H\"older continuous functions is denoted $C^\alpha([0,T];V)$. For any two-parameter function $A: \Delta_{[s,t]} \rightarrow V$ where
\[
    \Delta_{[s,t]} \defeq \{(u,v) \in [s, t]^2 ~ | ~ s \leq u < v \leq t\}
\]    
the difference operator $\delta A_{s,u,t}$ is defined by
$$
    \delta A_{s,u,t}:= A_{s,t}-A_{s,u}- A_{u,t},\ \ \ s \le u \le t,
$$
and the $\alpha$-H\"older norm is denoted
$$
\lvert A \rvert_{\alpha,[s,t]} \defeq \sup_{(u, v) \in \Delta_{[s, t]}} \frac{\lvert A_{u,v} \rvert}{|v - u|^\alpha}.
$$
The set of $\alpha$-H\"older continuous two-parameter functions is denoted $C^\alpha_2([0,T];V)$. For a random two-parameter $V$-valued function $A: \Omega \times \Delta_{[s,t]} \rightarrow V$ we set 
$$
\lVert A \rVert_{\alpha, m,[s,t]} \defeq \sup_{(u,v) \in \Delta_{[s, t]}} \frac{\lVert A_{u, v} \rVert_m}{\lvert v - u \rvert^\alpha}.
$$ 

The space of all functions for which the above norm is finite is denoted $C^\alpha_2([0,T];L^m)$ and $C^\alpha_2 L^m$ denotes the subspace of all $\mf \otimes \mb(\Delta_{[0,T]})/\mb(V)$-measurable elements of $C^\alpha_2([0,T];L^m )$. 

In the following, we usually omit the time interval in our notations and norms if $[s,t]=[0,T]$ and write `$\lesssim_{\theta}$' to indicate that an inequality holds up to a generic constant that depends on a parameter $\theta$.  

\begin{definition} 
\begin{itemize}
\item[$(1)$] A pair of (deterministic) continuous functions $\BBeta:=(\eta, \eta^{(2)}):[0,T] \times \Delta \rightarrow \R^d \times \R^{d \times d}$ is called a $\alpha$-H\"older continuous {\it rough path}, if the following two conditions are satisfied.
\[
(a)\  |\delta \eta |_{\alpha},\ |\eta^{(2)} |_{2\alpha} < \infty; \ \ (b)\  \eta^{(2)}_{s,t}- \eta^{(2)}_{s,u}- \eta^{(2)}_{u,t}= \delta \eta_{s,u} \otimes \delta \eta_{u,t},\ \text{ for any $s\le u \le t$.}
\]  
The space of $\alpha$-H\"older continuous rough paths is denoted  $\mc^{\alpha}([0,T],\R^d)$ and is equipped with the following seminorm and metric: for any $\BBeta=(\eta,\mn), \bBBeta=(\bar \eta, {\bar \eta}^{(2)} ) \in \mc^\alpha,$ 
$$
\vertiii{\bm \eta}_\alpha \defeq \lvert \delta \eta \rvert_\alpha +\sqrt{ \lvert \mn \rvert_{2 \alpha} }, \ \ \ 
\rho_{\alpha}(\bm \eta, \bm {\bar \eta}) = \lvert \delta \eta - \delta \bar \eta \rvert_\alpha + \lvert \eta^{(2)} - \bar \eta^{(2)} \rvert_{2\alpha}.
$$

\item[$(2)$.] Let $\mc^{0,\alpha}_g$ be the space of $\alpha$-H\"older geometric rough paths, i.e. the completion of canonically lifted smooth paths (i.e. if $\eta$ is smooth, $\BBeta:=(\eta, \int \eta  d\eta)$ is the canonical lift of $\eta$) under the rough path metric. Note that $\mc^{0,\alpha}_g$ is Polish and $\mc^{0,\alpha}_g \subset \mc^{\alpha}$ (see, e.g. \cite{friz2021roughito} for more information on this space).


\end{itemize}
\end{definition}

To introduce stochastic controlled rough paths and to define rough integrals, we introduce for any integrable random two-parameter process $A$ the two-parameter conditional expectations process $\E_\bullet A$ by
$$
(s,t;\ome)\mapsto \E[A_{s,t}|\mf_s](\ome) =: \E_s[A_{s,t}](\ome).
$$

\begin{definition}[{\cite[Definition 3.1]{friz2021roughito}}]\label{def:contro-rp} 
Let $\BBeta \in \mc^{\alpha}([0,T],\R^d)$ be a rough path.
\begin{itemize}

\item[$(1)$.] 
A pair \((Z,Z')\) is called a \textit{$\eta$-stochastic controlled rough path} of $p$-integrability and $(\beta,\beta')$-H\"older regularity with $\beta, \beta ' \in (0,\alpha],$ if the following holds: 

\begin{enumerate}  
\item[$(a)$] The process \(Z=(Z_t)_{t \in [0,T]}\) is $V$-valued, progressively measurable and 
\[
      \|Z_0\|_p < +\infty 
    \quad \text{and} \quad
    \|\delta Z\|_{\beta,p} := \sup_{0 \leq s < t \leq T} \frac{ \|\delta Z_{s,t} \|_p }{|t-s|^\beta} < +\infty.
\]
\item[$(b)$] The process \(Z'=(Z'_t)_{t \in [0,T]}\) is $\ml(\R^{d};V)$-valued\footnote{Since $\ml(\R^d;V)$ is finite dimensional in this paper, we omit its norm $|\cdot|$ in the following}, progressively measurable and
\[
      \|Z'_0\|_p < +\infty \quad \text{and} \quad \|\delta Z'\|_{\beta' ,p} := \sup_{0 \leq s < t \leq T} \frac{ \| \delta Z'_{s,t} \|_p }{|t-s|^{ \beta'}} < +\infty.
\]
\item[$(c)$] The two-parameter process 
\[
    R^Z_{s,t} := \delta Z_{s,t} - Z'_s \delta \eta_{s,t}, \quad (s,t)\in\Delta
\]
satisfies
\[
    \|\mathbb{E}_\bullet R^Z\|_{\beta+\beta', p } := \sup_{0 \leq s < t \leq T} \frac{\|\mathbb{E}_s R^Z_{s,t} \|_p}{|t-s|^{\beta+\beta'}} < + \infty.
\]
\end{enumerate}
In this case we write $(Z,Z') \in \mathbf{D}_\eta^{\beta,\beta'} L_{p}([0,T];V)$ and  $(Z,Z') \in \mathbf{D}_\eta^{2\beta} L_{p }([0,T];V)$ if $\beta = \beta'$. We equip the space $\mathbf{D}_\eta^{\beta,\beta'} L_{p}$ with the seminorm 
$$
\|(Z, Z')\|_{\eta, \beta,\beta', p} \defeq  \lVert \delta Z \rVert_{\beta, p} +   \left \lVert Z'_0 \right \rVert_{p} + \lVert \delta Z' \rVert_{\beta', p} + \lVert \E_\bullet R^Z \rVert_{\beta + \beta', p}.
$$ 
Moreover, for any $\BBeta, \bar\BBeta \in \mc^\alpha ,$ and pairs $(Z,Z') \in \md^{\beta, \beta'}_{  \eta} L_{p },$ $(\bar Z, \bar Z') \in  \md^{\beta, \beta'}_{\bar  \eta} L_{p }$ we define the distance
    \be
    \begin{split}
            d_{\eta, \bar \eta, \beta, \beta',p} \left( (Z,Z'), (\bar Z, \bar Z') \right):= & \ \|\delta (Z - \bar Z)\|_{\beta,p } + \|\delta (Z' - \bar Z')\|_{\beta',p } + \|Z_0- \bar Z_0\|_p  \\
            & + \|Z'_0- \bar Z'_0\|_p + \|E_\bullet R^Z- E_\bullet \bar R^{\bar Z}\|_{\beta+ \beta',p},
    \end{split}
    \ee
    where $$\bar R^{\bar Z}_{s,t}:= \delta \bar Z_{s,t} - \bar Z'_s \delta \bar \eta_{s,t}, \quad s,t \in \Delta.$$

\item[$(2)$] For any pair $(Z,Z') \in \mathbf{D}_\eta^{\beta,\beta'} L_{p},$ if $R^{Z} \in C^{\beta+\beta'}_2 L_p$, we write $(Z,Z') \in \D^{\beta,\beta'}_{\eta} L_p$, and define for any $(\bZ, \bZ') \in \D_{\bar{\eta}}^{\beta,\beta'} L_{p}$, the seminorm
\begin{align*}
    \vertiii{ Z,Z'; \bar Z, \bar{Z}' }  & ~ := \|\delta (Z - \bar Z)\|_{\beta,p } + \|\delta (Z' - \bar Z')\|_{\beta',p } \\ & \quad \quad + \|Z_0- \bar Z_0\|_p + \|Z'_0- \bar Z'_0\|_p + \|  R^Z-   \bar R^{\bar Z}\|_{\beta+ \beta',p}.
\end{align*}
\end{itemize}
We may omit the lower index $\eta$ whenever there is no risk of confusion.
\end{definition} 


The following lemma establishes useful estimates for controlled rough paths that will be used throughout. The proof is not difficult and is available on request.

\begin{lemma}\label{productest}
    Let $(Z,Z') \in \md^{\beta,\beta'}_\eta L_{\infty} $ and $(z,z') \in \md^{\beta,\beta' }_\eta L_{m}$ for some $m \in [2,\infty]$. Then 
    $$
    (Zz,(Zz)' ):= (Zz, Z'z+Zz') \in \md^{\beta,\beta' }_\eta L_{m}.
    $$
    Moreover, if
    $$
        \esssup_{t,\ome}|Z_t| \vee   \|(Z,Z')\|_{\eta, \beta,\beta',\infty}   \le K
        \quad \mbox{and} \quad \|z_0\|_m \vee \|z'_0\|_m \le K_0
    $$
    then, for some implied constants that depend increasingly on $T$, 
\be
\begin{split}\label{Fy}
   & \|Zz\|_{ \beta,m} \lesssim K\|z\|_{ \beta,m} + KK_0,\\ 
  &  \|(Zz)'\|_{ \beta',m} \lesssim K(\|z\|_{ \beta,m}+\|z'\|_{ \beta',m}   + K_0), \\  
  & \|   R^{Zz}\|_{ \beta ,m} \lesssim K (\|z\|_{ \beta,m}+\|z'\|_{ \beta',m}  +  K_0)(1\vee \|  \eta\|_{\beta}), \\ 
  &  \| \EB R^{Zz}\|_{ \beta+\beta',m} \lesssim K (\|z\|_{ \beta,m}+ \| \EB R^z\|_{ \beta+\beta',m}  + K_0),
\end{split}
\ee
where $a \vee b:= \max{(a,b)}$.
\end{lemma}

Having introduced rough paths, we proceed to introduce rough stochastic integrals. The next proposition follows by \cite[Theorem 3.5]{friz2021roughito} with $m=n=p$.  We give a proof in the appendix to keep the paper self-contained.

\begin{prop}[{\cite[Theorem 3.5]{friz2021roughito}}] \label{prop:roughinteg}
Let $\BBeta \in \mc^{\alpha}([0,T],\R^d)$ be a rough path and $\beta  \in (0,\alpha],$ $\beta' \in (0,\beta],$ such that $\alpha+ \beta > \frac12$ and $\alpha+\beta + \beta' > 1$. For any $p \in [2,\infty ]$ and any $\eta$-stochastic controlled rough path $(Z,Z') \in \BD_\eta^{\beta,\beta'}L_p$, the following rough stochastic integral is well-defined:
\be
  \int_0^t Z_r d\BBeta_r :=  \int_0^t (Z_r,Z'_r) d\BBeta_r :=  \lim_{|\Pi| \rightarrow 0} \sum_{[u,v] \in \Pi} ( Z_u \delta \eta_{u,v} + Z'_u \mn_{u,v}), \quad t\in[0,T].
\ee
Here, $\Pi$ is any partition of $[0,t]$ with mesh $|\Pi|$ and the limit is taken in the sense of convergence in probability. Moreover, for any $(s,t) \in \Delta,$ $q \le p $ and $q<\infty$,
\begin{align*} 
    &\| \int_s^t Z_r d\BBeta_r - Z_s \delta \eta_{s,t} \|_q \lesssim    (\vertiii{\BBeta}_\alpha \| \delta Z \|_{\beta,p} +\vertiii{\BBeta}_\alpha^2 \sup_{r\in [s,t]} \|Z'\|_p )|t-s|^{\alpha+\beta} \\  
    & \qquad \qquad \qquad \qquad \qquad \ \ \  +   ( \vertiii{\BBeta}_\alpha \|\EB R^Z\|_{\beta+\beta',p} + \vertiii{\BBeta}_\alpha^2 \| \EB \delta Z'\|_{\beta',p}) |t-s|^{\alpha+\beta+\beta'},\\  
   & \lVert \mathbb{E}_{s}  ( \int_{s}^{t} Z_{r} d\BBeta_r - Z_{s} \delta \eta_{s,t} - Z'_s \mn_{s,t}  ) \rVert_{q} 
    \lesssim   \left(  \vertiii{\BBeta}_{\alpha} \|\EB R^{Z} \|_{ {\beta}+ \beta', p} + \vertiii{\BBeta}_{\alpha}^2 \| \EB \delta Z' \|_{\beta',p} \right) |t-s|^{\alpha +   \beta + \beta'}.
\end{align*}
In particular, 
$$
    (X,X'):=(\int_0^. Z_r d \BBeta_r, Z) \in \D^{\alpha, \beta}_\eta L_q.
$$
\end{prop}


The boundedness of a path $(Z,Z') \in \BD^{\alpha,\beta}_\eta L_\infty$ does not guarantee the boundedness of $(\int Z d\BBeta, Z)$. Boundedness is guaranteed if the rough path belongs to the space $\D^{\alpha,\beta}_\eta L_\infty$ as shown by the following lemma. 

\begin{lemma}\label{lem:gub-sew}
Suppose that $(Z,Z') \in  \mathscr{D}^{\beta,\beta'}_\eta L_p $ with $p \in [2,\infty],$ $\beta \in (0,\alpha]$, $\beta' \in (0,\beta],$ and $\alpha+\beta+ \beta'>1.$ Then for any $0 \le s \le t \le T,$ we have 
$$
\| \int_s^t Z_r d\BBeta_r - Z_s \eta_{s,t} - Z'_s \mn_{s,t}  \|_{\infty} \lesssim_{\alpha,\beta,\beta'} \vertiii{\BBeta}_{\alpha} ( \|\delta Z'\|_{\beta',\infty} + \|R^Z\|_{\beta+\beta', \infty} ) (t-s)^{\alpha+\beta+\beta'}.
$$
In particular, if $p=\infty,$ then $$\Big(\int_0^\cdot Z_r d \BBeta_r, Z\Big) \in \mathscr{D}^{\alpha,\beta}_\eta L_\infty.$$
\end{lemma}

\begin{proof}
Let $A_{s,t}:= Z_s \eta_{s,t} + Z'_s \mn_{s,t} $. Then $|\delta A_{s,u,t}|= |R^Z_{s,u} \delta \eta_{u,t} + (\delta Z'_{s,u}) \eta^{(2)}_{u,t}|$, and so 
$$
    \| \delta A_{s,u,t} \|_{L_p} \lesssim (\|R^Z\|_{p,\beta+\beta'}+ \|\delta Z'\|_{p, \beta'} ) (\vertiii{\BBeta}_{\alpha} \vee 1)\vertiii{\BBeta}_{\alpha} (t-s)^{\alpha +\beta+ \beta' }.
$$    
Thus, the assertion follows from Gubinelli's sewing lemma; see, e.g. \cite{Gub04} or \cite[Lemma 4.2]{frizroughpaths2021} for details.   
\end{proof}


The following rough It\^o formula is taken from \cite[Proposition 2.15]{BCN24}.

\begin{prop}\label{roughito}
Let $(\beta, \beta')$ be as in Proposition \ref{prop:roughinteg}, $X$ be a progressively measurable process, $\BBeta \in \mc^{0,\alpha}_g$ be a geometric rough path, and $(X',X'')$ be a $\eta$-stochastic controlled rough path such that a.s. 
$$
X_t = X_0 + \int_0^t b_s d s + \int_0^t \sigma_s d W_s + \int_0^t X'_s d \BBeta_s, \quad t \in [0,T].
$$
Suppose that $(X',X'')$ belongs to $\BD^{\beta,\beta'}_{\eta} L_p$ for any $p\ge 2$, and that $(b,\sigma)$ are progressively measurable processes that satisfy 
     \begin{equation} \label{eq:Lp-coefficients}
         \sup_{t\in [0,T]}\|b_t\|_{m}, 
         \sup_{t\in [0,T]}\|\sigma_t\|_{m} < \infty, \ \ \ \text{ for any } \ \  m \ge 2.
     \end{equation}
Let $f$ be a three times continuously differentiable function such that for $k=1,2,3$
\begin{equation*}
        |D^kf(x)| \leq C(1+|x|^q),
\end{equation*}
for some $q\geq1$ and $C>0$. Then, for every $t\in [0,T]$, $\mathbb{P}$-a.s.,
\begin{equation}\label{eq:roughito}
\begin{split}
            f(X_t) - f(X_0)
            = & \int_{0}^{t} Df(X_u) b_u d u
            + \int_{0}^{t} Df(X_u) \sigma_u d W_u
            + \int_{0}^{t} (Y,Y^{\prime})_u d \BBeta_u\\
            &\  + \frac{1}{2} \int_{s}^{t} \operatorname{trace}(\sigma_u \sigma^{\top}_u D^2f(X_u)) d u,
\end{split}
\end{equation}
where 
$
    (Y,Y^\prime) := (Df(X)X^{\prime}, D^2f(X)(X^{\prime})^2 + Df(X)X^{''}) \in \bigcap_{p \ge 2} \BD_\eta^{\beta, \beta' } L_p.
$
\end{prop}


\section{Well-posedness of rough SDEs with affine drivers}

In this section, we establish the well-posedness of rough SDEs with affine rough drivers, as well as some key properties that we need to introduce rough and pathwise linear-quadratic control problems. Specifically, we consider the $\R^{d_X}$-valued rough SDE
\begin{equation} \label{eq:linear_forward_rsde}
    dX_t = b_t(X_t, \BBeta)dt + \sigma_t(X_t, \BBeta) dW_t + (F_t(\BBeta) X_t + f_t(\BBeta) ) d\bm{\eta}_t, \quad X_0 = \xi
\end{equation}
where $W$ is a $d_W$-dimensional Brownian motion, $\xi \in \MF_0$, and $\bm{\eta} = (\eta, \mn) \in \mc^\alpha([0,T],\R^d)$. The coefficient processes
$$
(b,\sigma) ( \BBeta):[0,T] \times \Omega \times \R^{d_X}   \rightarrow \R^{d_X} \times \R^{d_X \times d_W}
$$ 
are assumed to be progressively measurable, and the stochastic controlled rough paths $(F,F')(\BBeta)$ and $(f,f')(\BBeta)$ are assumed to belong to the spaces  $\BD_\eta^{\beta,\beta'}L_{\infty}$ and  $\BD_\eta^{\beta,\beta'}L_{m}$, respectively.

\begin{definition}\label{def:solution}
We call a continuous adapted process $X$ an $m$-integrable solution to the rough SDE \eqref{eq:linear_forward_rsde} in the space $\BD_{\eta}^{\beta,\beta'}L_m([0,T]; \R^{d_X})$, if the following holds:

\begin{itemize}
\item[$(a)$] The integral $\int_0^T |b_t(X_t,\BBeta)| dt + \int_0^T |\sigma_t \sigma^{\mathrm{T}}_t(X_t,\BBeta)|dt$ is finite a.s. 

\item[$(b)$] The process $\Big(F(\BBeta)X+f(\BBeta), F'(\BBeta)X+F(\BBeta)(F(\BBeta)X+f(\BBeta))+ f'(\BBeta)\Big)$ is a stochastic controlled rough path and belongs to $ \BD_\eta^{\beta,\beta'}L_m$ for some $m \ge 2, \beta \in (0,\alpha]$ and $\beta' \in (0,\beta],$ with 
$$
    \alpha+\beta>\frac12 \quad \mbox{and} \quad \alpha+\beta+\beta'>1.
$$
\item[$(c)$]  The following integral equation holds a.s.~for all $r \in [0,T]$:
$$
X_r= \xi + \int_0^r b_t(X_t,\BBeta)dt + \int_0^r \sigma_t(X_t,\BBeta) dW_t + \int_0^r (F_t(\BBeta) X_t + f_t(\BBeta)) d\bm{\eta}_t 
$$ 
\end{itemize}
\end{definition}

We refer to $(b,\sigma,F,F',f,f')$ as the {\sl coefficients} of the rSDE \eqref{eq:linear_forward_rsde}. To prove the existence and uniquness of solutions the dependence of the parameters on $\BBeta$ is not relevant. However, as we will see, defining and solving anticipative control problems requires $\BBeta$-dependent parameters. This is because the optimal control and hence the coefficients of the resulting rough SDE will depend on $\eta$; see Section \ref{sec:rSMP} for details.


\subsection{Existence and uniqueness  of solutions}

Rough SDEs with bounded parameters have been considered in \cite{friz2021roughito}. We establish the existence of a unique solution for unbounded coefficients under the following conditions. Since the rough path $\BBeta$ is fixed throughout this subsection, we omit it to simplify the notation. 

\begin{Assumption}\label{ass:m-linear-growth_lip} Assume that $m\in [2,\infty),$ $\beta \in  (\frac{1}{4}, \alpha]$, $\beta' \in (0, \beta]$, and $\xi \in L_m.$ Moreover, assume that:
\begin{itemize}
    \item[$(i)$] For $\psi \in \{b, \sigma\}$, there exists $K^\psi \in H^{m,n^{\psi}},$ with
    $$
    n^b \geq \frac{1}{1 - (  \alpha + \beta)} \vee m,\ \ \  n^\sigma \geq \frac{2}{1- 2\beta}  \vee m,
    $$ 
such that almost surely for any $(t,x, \bar x) \in [0,T] \times \R^{d_X} \times \R^{d_X} ,$
\begin{equation}\label{ine:Kg}
    \begin{split}
       \lvert \psi(  t, \ome, x) \rvert \lesssim  (\lvert K^{\psi}_t(\ome) \rvert + \lvert x \rvert) ; \ \ \    \lvert \psi(  t, \ome, x) - \psi(  t, \ome, \bar x)\rvert \lesssim  \lvert x - \bar x \rvert.
    \end{split}   
\end{equation}
\item[$(ii)$] The stochastic controlled rough path $(F, F')$ is an element of $  \md^{\beta,\beta'}_\eta L_{\infty}$, where the pair $(   \beta,\beta')$ satisfies 
$\bar \alpha  + \beta + \beta' > 1$ with $\bar \alpha:= (\frac12-\frac{1}{n^\sigma}) \wedge \alpha$. In particular, 
\begin{equation}\label{ine:KF}
  (\esssup_{t,\ome}|F_t|  ) \vee   \|(F,F')\|_{\eta, \beta,\beta',\infty}   \le K
\end{equation}
for some constant $K$. Moreover, the stochastic controlled rough path $(f, f')$ belongs to $ \md^{\beta,\beta' }_\eta L_{m}$. 

\end{itemize}
\end{Assumption}

\begin{rem} 
In part (i) of the above assumption, $(n^b,n^\sigma)$ depends on $(\beta, \beta')$ while in (ii) $(\beta,\beta')$ depends on $n^\sigma.$ If $(b,\sigma)$ is bounded, we can take $(n^b,n^\sigma)$ as any positive number. In this case, $\bar \alpha=\alpha$ and the constraint on $(\beta,\beta')$ in (ii) reads $\alpha+\beta+\beta'>1.$ If $(b,\sigma)$ is unbounded, we require $\bar \alpha \ge \beta.$ This holds if $2\beta + \beta'>1.$ 
 \end{rem}


In what follows we set 
$$
    \Theta:=(\alpha,\beta,\beta', n^b, n^{\sigma}, T).
$$

\begin{example}\label{exam:controlledrp}
Let $\bm{\eta} = (\eta, \mn) \in \mc^\alpha$ be a rough path. The following are two examples where our assumptions on the stochastic controlled rough paths $(F,F')$ and $(f,f')$ can easily be verified. 

\begin{itemize}

\item [$(1)$] Let $G:[0,T] \times \Omega \times \R^{d} \rightarrow \R^k$ be a function such that:
\begin{itemize}
    \item for any $x\in \R^d,$ the process $G(\cdot, x)$ is progressively measurable;
    \item for any $(t,\ome) \in [0,T] \times \Omega$, the function  $G(t,\ome, \cdot)$ belongs to $C^3_b (\R^{d}; \R^{k})$, and all derivatives are continuous in $(t,x)$ for any $\ome;$
    \item for any $(\ome,x)\in \Omega \times \R^d$, $ G(\cdot, \ome,x)$ belongs to $C^{\gamma}$ with $\gamma>1-\alpha,$ and $| G(\cdot,\ome,x) |_{\gamma}$ is uniformly bounded in $(\ome,x)$.
\end{itemize}
Then 
$
    (F,F')(t,\ome ):=(G(t,\ome,\eta_t), D_x G(t,\ome, \eta_t)) \in \md^{2 \beta }_\eta L_{\infty}
$ 
with $\beta := \alpha \wedge \frac{\gamma}{2}$. Indeed, 
\[
\begin{split}
    G(t,\ome,\eta_t)- G(s,\ome, \eta_s)= & \ D_x G(s,\ome, \eta_s) \delta \eta_{s,t}  + ( G(t,\ome, \eta_t) -  G(s,\ome, \eta_t) ) \\
    & + \int_0^1 \int_0^1  D_x^2G(t,\ome, \eta_s+ \lambda \mu \delta \eta_{s,t} ) d\lambda d\mu (\delta \eta_{s,t})(\delta \eta_{s,t})  
\end{split}
\]
and it is not difficult to check that $\EB R^F \in C_2^{2\alpha \wedge \gamma} ([0,T]; L_{\infty})$ where 
$$
    R^F_{s,t}:= ( G(t,\ome, \eta_t) -  G(s,\ome, \eta_t) )   + \int_0^1 \int_0^1  D_x^2G(t,\ome, \eta_s+ \lambda \mu \delta \eta_{s,t} ) d\lambda d\mu (\delta \eta_{s,t})(\delta \eta_{s,t})
$$ 
Similar examples can be constructed for $(f,f')$ under suitable integrability conditions on $D^i_x G$. 

\item [$(2)$] Let $H \in C^2_b(\R^d;\R)$ and 
$
    (f,f')(t,\ome):=(H(W_t)(\ome), 0). 
$    
It then follows from It\^o's formula that $(f,f') \in \md^{ 2 \alpha  }_\eta L_{p}$ for any $p \in [2,\infty).$ Indeed, $f,f' \in C^{\frac12}([0,T];L_p )$ and $\EB R^f \in C^{1}_2 L_p $.
\end{itemize}
\end{example}

To establish the well-posedness of our rSDE we require the following a priori estimates.

\begin{prop}\label{priorest-lrsde}
Suppose that Assumption \ref{ass:m-linear-growth_lip} holds and that $X$ is a solution in the sense of Definition \ref{def:solution}, to the rough SDE \eqref{eq:linear_forward_rsde} in the space $\BD_\eta^{\bar \beta, {\bar \beta}'}L_m$ for some $\bar \beta \in (0,\beta],$ ${\bar \beta}' \in (0,\beta']$. Then $X$ is a solution in the space $\md^{\bar \alpha,\beta}_\eta L_{m}$ with $\bar \alpha:= \alpha \wedge (\frac12- \frac{1}{n^\sigma}) \ge \beta$. Moreover, setting $X' := FX+f, $ it holds that 
\begin{equation}\label{est-sol}
    \|(X,X')\|_{ \bar \alpha, \beta,m} + \| X\|_{S^m } \le  2^{C_{\Theta} N}\left( \lVert \xi \rVert_m +  {\theta}({\frac1K}\wedge1) \right),
\end{equation}
where $C_{\Theta}$ is a constant depending only on $\Theta,$ $K$ is the bound introduced in Assumption \ref{ass:m-linear-growth_lip}, 
$$
    N:=\left((K^2 \vee 1)( \vertiii{\BBeta}_{\alpha} \vee 1) \right)^{\frac{1}{\beta'}}+1 \quad
    \mbox{and} 
    \quad \theta \defeq \lVert K^b \rVert_{H^{m, n^b}} + \lVert K^\sigma \rVert_{H^{m, n^\sigma}} + \|f_0\|_m + \|(f, f')\|_{\beta,\beta', m}.
$$    
\end{prop}

\begin{proof}
Let $ (Z , Z') \defeq  (FX+f, (FX)'+f') $ and $  
    J_{s , t } := \int_s^t Z_r d \bm \eta_r - Z_s \delta \eta_{s, t} - Z'_s  \eta^{(2)}_{s, t}.$ Moreover, let $I_{s,t}:= \int_s^t Z_r d \BBeta_r$ and $R^I_{s,t}:= I_{s,t}- Z_s \delta \eta_{s,t}.$   
In terms of this notation,
	\be\label{eq:xN_bound}
        {\delta X}_{s, t} =  J_{s, t} + \int_{s }^{t } b_r(X_r)dr +  \int_{s }^{t } \sigma_r(X_r) d W_r + (Z_s \delta {\eta}_{s, t} + Z'_s \eta_{s, t}^{(2)}  ).
\ee
Moreover, we define $M:=\vertiii{\BBeta}_{\alpha,[0,T]}$, $\Gamma_{s, t} \defeq \sup_{r \in [s, t]} \lVert X_r \rVert_m$ and\footnote{By the following steps it is easy to check that $(\Xi_{s, t},\Psi_{s, t})$ are finite. Indeed, the following steps hold with $(\bar \alpha, \beta)$ replaced by $(\bar \beta, {\bar \beta}')$. Then one sees the finiteness by {\bf Step 1}.}
\begin{align*}
    (\Xi_{s, t},\Psi_{s, t})  & \defeq ( \lVert \delta X \rVert_{\balpha, m; [s, t]} ,   \lVert \E_\bullet R^{X} \rVert_{\balpha + \beta , m; [s, t]}).
\end{align*}

We establish the global estimate \eqref{est-sol} in two steps. First, we establish a bound on $\|(X,X')\|_{\bar{\alpha}, \beta, m}$ in terms of the quantities $(J, \theta, \Gamma, \Xi, \Psi)$. 
Proposition \ref{prop:roughinteg} yields a local-in-time bound on $J$ that depends on 
$(\Gamma, \Xi, \Psi)$ and $\theta$. Substituting this back gives 
a local bound for $\|(X,X')\|_{\bar{\alpha}, \beta, m}$ in terms of $(\theta, \xi)$, which is then 
extended globally via a telescoping argument. Second, $\|X\|_{S^m}$ is controlled using the 
Stochastic Sewing Lemma \cite[Theorem 2.8]{friz2021roughito}.

    
{\bf Step 1.} {\sl Bounding $\|  (X,X') \|_{\balpha,\beta, m}$ in terms of $(J,\theta, \Gamma,\Xi, \Psi)$.} 
For any $[s, t] \in \Delta_{[0, T]}$, it holds that
    \begin{equation}
    \begin{split}
          \lVert \int_s^t b_r(X_r) dr   \rVert_m &\leq \lVert \int_s^t b_r(X_r)dr -  \int_s^t b_r(0)dr \rVert_m + \lVert   \int_s^t b_r(0)dr \rVert_m \\ 
        &\lesssim_T  (\int_s^t  (\sup_{u \in [s, t]} \left \lVert X_u \right \rVert_m^m  )dr  )^\frac{1}{m} (t-s)^{1 - \frac{1}{m}}  +  \E \left[ \int_s^t \lvert K^b_r \rvert^m dr \right]^{\frac{1}{m}} (t - s)^{1 - \frac{1}{m}} \\ \label{bound-b}
        &\leq \Gamma_{s, t}(t - s) +   \lVert K^b  \rVert_{H^{m, n^b}} (t - s)^{1 - \frac{1}{n^b}}
    \end{split}
    \end{equation}
    where the last inequality uses H\"older's inequality and Assumption \ref{ass:m-linear-growth_lip}.
  Similarly, using the Burkholder-Davis-Gundy inequality, we see that
    \begin{equation}\label{eq:sigma_bound}
         \lVert \int_s^t \sigma_r(X_r) dW_r   \rVert_m \lesssim_T \Gamma_{s, t}(t - s)^\frac{1}{2} + \left \lVert K^\sigma \right \rVert_{H^{m, n^\sigma}}(t - s)^{\frac{1}{2} - \frac{1}{n^\sigma}}.
    \end{equation}
    Furthermore,
    \begin{equation*}
        \left \lVert Z_t \right \rVert_m  \leq \lVert F_t X_t \rVert_m + \lVert f_t \rVert_m \leq K \left \lVert X_t \right \rVert_m + \theta
    \end{equation*}
    and similarly,
    \begin{align}\label{eq:Z'_bound}
    \begin{split}
        \lVert Z'_t   \rVert_m  \leq \lVert F_t (F_t X_t + f_t) \rVert_m + \lVert F'_t X_t + f'_t \rVert_m \leq  K(K+1) \lVert X_t   \rVert_m +  (K + 1) \theta.
    \end{split}
    \end{align}
As a result, for some implicit constant that depends on $\Theta$, 
    \begin{equation*}
         \lVert Z_s \delta \eta_{s, t} + Z'_s \mn_{s, t}   \rVert_m \lesssim M (1+K) (K \Gamma_{s,t} + \theta)(t - s)^\alpha.
    \end{equation*}

In view of \eqref{eq:xN_bound} and the preceding estimates, we conclude that
    \begin{equation}\label{eq:delta_x_bound}
        \left \lVert \delta X_{s, t} \right \rVert_m \leq \left \lVert   J_{s, t} \right \rVert_m + MK(1+K) \Gamma_{s,t}(t - s)^\alpha + (M(1+K) \vee 1)\theta(t - s)^{\balpha}.
    \end{equation}

Since $X'=Z,$ we need to estimate $\|Z\|_{\beta,[s,t]}.$ Since 
\begin{equation}\label{deltaZ}
        \lvert \delta Z_{s, t} \rvert  \le \lvert   F_t \delta X_{s,t} \rvert + \lvert  \delta F_{s,t} X_s \rvert +|\delta f_{s,t}| \leq  K  |\delta X_{s,t}| + K (t-s)^\beta |X_s| + | \delta f_{s,t}|,  
\end{equation}
we see that 
\begin{equation}\label{eq:deltaZbound}
        \lVert  \delta Z \rVert_{\beta; m; [s, t]} \lesssim   K\Xi_{s,t} + K \Gamma_{s,t} + \theta.
\end{equation}

Finally,
    \begin{equation*}
    	\E_s R^{ X}_{s, t} = \E_s \left(\delta X_{s, t} - Z_s \delta \eta_{s, t} \right) = \E_s  \left[  J_{s, t} + \int_{s }^{t } b_r(X_r) dr   +  Z'_s \mn_{s, t}    \right]
    \end{equation*}
and with $\epsilon' \defeq \min\{1 - \frac{1}{n^b}, 2\alpha  \}$ it follows by \eqref{bound-b} and \eqref{eq:Z'_bound} that
\begin{equation}\label{eq:RXN}
    	\lVert \E_s R^{ X}_{s, t} \rVert_m \lesssim \lVert \E_s   J_{s, t} \rVert_m + (1\vee M)(K+1)  K  \Gamma_{s, t} (t - s)^{2\alpha  } + (1\vee M)(K+1)   \theta (t - s)^{\epsilon'}.  
\end{equation}



{\bf Step 2.} 
{\sl Estimate $\|(Z,Z')\|_{\beta,\beta'}$.} The following bounds will be used below. 
By Lemma \ref{productest},
\begin{align*}
		\|\delta Z'\|_{\beta',m,[s,t]} \le \| \delta (FX)'\|_{\beta',m,[s,t]} +\| \delta f'\|_{\beta',m,[s,t]} \lesssim K( \| \delta X\|_{\beta,m,[s,t]} + \| \delta X'\|_{\beta',m,[s,t]} + \Gamma_{s,t}) + \theta
\end{align*}
and so it follows from \eqref{eq:deltaZbound} that 
\begin{equation}\label{est-z'}
    \| \delta Z'\|_{\beta',m,[s,t]} \lesssim (1+K) ( K \Xi_{s,t} + K \Gamma_{s,t} + \theta ).
\end{equation}
Since $R^{Z}= R^{FX} + R^{f},$ applying Lemma \ref{productest} again, yields that
\begin{equation}\label{est-rz}
\begin{split}
&\|R^Z\|_{\beta,m,[s,t]} \lesssim (1 \vee M) (1 \vee K)(K\Xi_{s,t}+ K\Gamma_{s,t}  
+  \theta),\\
& \|\E_{\bullet}R^Z\|_{\beta+\beta',m,[s,t]}    \lesssim K( \Xi_{s,t} + \Psi_{s,t} + (1+K) \Gamma_{s,t}  )+ (1\vee K)\theta.    
\end{split}
\end{equation}

    
{\bf Step 3.} {\sl Bounding $J_{s, t}$ - and thus $\|(X,X')\|_{\balpha, \beta,m,[s,t]}$ - 
in terms of $(\theta, \Gamma,\Xi, \Psi).$} By \cite[Theorem 3.7]{friz2021roughito} (or Proposition \ref{prop:roughinteg} above), 
    \begin{equation}
    \begin{split}\label{eq:Jbound}
          	\lVert   J_{s, t} \rVert_m \lesssim  & \  
        \Big(M \lVert \delta Z \rVert_{\beta, m, [s, t]} + M^2 \sup_{r \in [s, t]} \lVert Z'_r \rVert_m \Big) \lvert t - s \rvert^{\alpha + \beta} \\
           & + \Big( M \lVert \E_\cdot R^{  Z} \rVert_{\beta + \beta', m, [s, t]} + M^2 \lVert \delta Z' \rVert_{\beta', m, [s, t]}  \Big) (t - s)^{\alpha + \beta + \beta'}. 
    \end{split}
    \end{equation}
In view Step 2, 
	\begin{equation*}
		\lVert   J_{s, t} \rVert_m \lesssim {K_1}   \left(M\theta  + M K\Gamma_{s, t}  + \Xi_{s, t} \right) \lvert t - s \rvert^{\alpha + \beta} + {K_1} \left(M \theta + M K \Gamma_{s, t} + MK \Xi_{s, t} +  \Psi_{s, t} \right) (t - s)^{ \alpha + \beta + \beta'}
	\end{equation*}
 where $K_1:= (1\vee K)(1 \vee M).$ Substituting the above into \eqref{eq:delta_x_bound}, and assuming $|t-s|^{\beta'}K_1 <1$, we have
	\begin{equation}\label{eq:deltaxgammabound}
		\lVert \delta X_{s, t} \rVert_m \lesssim K_1 \theta (t - s)^{\balpha} + K_1 K \Gamma_{s, t}(t - s)^{\alpha} + K_1  \Xi_{s, t} (t - s)^{\alpha + \beta} + K_1  \Psi_{s, t}(t - s)^{\alpha + \beta + \beta'}.
	\end{equation}		
Using again Theorem \cite[Theorem 3.7]{friz2021roughito} yields that
	\begin{equation*}
 \begin{split}
		\lVert \E_s   J_{s, t} \rVert_m \lesssim & \  (M \lVert \E_\bullet R^{  Z} \rVert_{\beta + \beta', m, [s, t]} + M^2 \lVert \delta Z' \rVert_{\beta', m, [s, t]}) ( t - s )^{\alpha + \beta + \beta'} \\
 \lesssim & \ K_1  (M\theta + MK \Gamma_{s, t} + MK \Xi_{s, t} +  \Psi_{s, t} ) ( t - s )^{\alpha + \beta + \beta'}
 \end{split}
	\end{equation*}
 where we applied \eqref{est-rz} and \eqref{est-z'} in the last inequality. In view of \eqref{eq:RXN} this shows that
	\begin{equation}\label{eq:Rgammabound}
		\lVert \E_s R^{ X}_{s, t} \rVert_m \lesssim K_1 \theta (t - s)^{\alpha+\beta} + K_1 K \Gamma_{s, t}(t - s)^{2\alpha}  + K_1  (MK \Xi_{s, t} + \Psi_{s, t}) (t - s)^{\alpha + \beta + \beta'}
	\end{equation}

 
{\bf Step 4.} {\sl Bounding $\Gamma$, $\Xi$ and $\Psi$ in terms of $(\theta,\xi)$.} Using that $\alpha \ge \balpha $ and $|t-s|^{\beta'}K_1 < 1$, taking the $(\balpha + \beta)$-H\"older norm on both sides of the inequality \eqref{eq:Rgammabound} we obtain
	\begin{equation}\label{psibd}
		\Psi_{s, t} \le C   ( K_1  \theta + K_1  K \Gamma_{s, t} + K_1^2    |t-s|^{\beta'} \Xi_{s, t} + K_1 (t-s)^{\beta'}  \Psi_{s,t} ).
	\end{equation}
 where $C $ is a constant depending only on $\Theta$ that may change from line to line. Then, for any $\delta_1>0$ that satisfies $C  K_1  \delta_1^{\beta'} \le \frac12$  and any $t-s< \delta_1,$ we have 
 \be\label{psi-new}
 \Psi_{s, t} \le   C  K_1(\theta + K \Gamma_{s, t} ) + K_1 \Xi_{s, t}
 \ee
Substituting this into equation \eqref{eq:deltaxgammabound}, and we obtain for any $t-s< \delta_1$ that
	\begin{equation*}
		\lVert \delta X_{s, t} \rVert_m \lesssim K_1 \theta(t - s)^{\balpha} + K_1 K \Gamma_{s, t}(t - s)^{\alpha} + K_1  \Xi_{s, t} (t - s)^{\alpha + \beta}.
	\end{equation*}
    
According to the above estimate, for any $t-s< \delta_2$ with any $\delta_2>0$ satisfying $C  K_1  \delta_2^{\beta}\le  \frac12$, we have  
	\begin{equation}\label{xibd}
		\Xi_{s, t} \le C  {K_1} (\theta + K \Gamma_{s, t}).
	\end{equation}

We now substitute \eqref{psi-new} and \eqref{xibd} into \eqref{eq:deltaxgammabound}, to obtain, for any $t-s< \delta_3$ with $C  K_1  \delta_3^{\beta}\le  \frac12$,
	\begin{equation}\label{eq:deltaxfinalbound}
	 	\lVert \delta X_{s, t} \rVert_m \le C  K_1 (\theta + K \Gamma_{s, t})(t - s)^{\balpha}.
	\end{equation}
 
Next, we choose any $\delta_4 >0$ such that $C  K_1 K \delta_{4}^{\balpha} \le \frac12,$ and a partition  $\mpe = \{0 = t_0 < t_1 < \dots < t_N = T\}$ with mesh $|\mpe| < \delta_0:= \mathrm{min}_i \{\delta_i\}_{i=1,...,4} \wedge 1,$ and $N\lesssim_{\Theta} (K_1 \vee K K_1)^{\frac{1}{\beta'}}+1$. It then follows from \eqref{eq:deltaxfinalbound} that 
 \begin{equation*}
		\Gamma_{t_i, t_{i+1}} \le   \lVert X_{t_i} \rVert_m + C  K_1 \theta \delta_0^{\balpha} + \frac12 \Gamma_{t_i, t_{i+1}},
	\end{equation*}
 for any $i=0,...,N-1$ from which we conclude that 
 \begin{equation*}
     \Gamma_{t_i, t_{i+1}} \le  2 \rVert X_{t_i} \rVert_m + \theta  K'. 
 \end{equation*}
 with $K':=\frac{1}{K} \wedge 1.$
 Since $\Gamma_{t_0, t_{1}} \le 2 \lVert \xi \lVert_m + {\theta}{K'}$ and $\lVert X_{t_i} \rVert_m \leq \Gamma_{t_{i-1}, t_i}$ a telescopic argument yields
\begin{equation*}
    \Gamma_{0,T} = \sup_i \Gamma_{t_i,t_{i+1}} \le 2^N (\lVert \xi \rVert_m + {\theta}{K'} ).
\end{equation*}

In view of the above bound and  \eqref{eq:deltaxfinalbound}, we have for any $s,t$ with $|t-s|<\delta_0,$
\begin{equation}\label{fbd-xi}
\Xi_{s,t } \lesssim 2^{CN} (\lVert \xi \rVert_m + {\theta}{K'} ).
\end{equation}
 Noting that $\|\delta X\|_{\beta, m,[0,T]} \le  \sum_i \|\delta X\|_{\beta, m,[t_i,t_{i+1}]}$, we obtain 
\begin{equation}\label{fbd-x}
 \Xi_{[0,T]} =  \| \delta X\|_{\beta, m,[0,T]} \lesssim 2^{CN}( \lVert \xi \rVert_m + {\theta}{K'}).
\end{equation}

By \eqref{psi-new} a similar argument yields that  
$$
\Psi_{t_i,t_{i+1}} \lesssim 2^{CN} ( \lVert \xi \rVert_m + {\theta}{K'} )
$$

Likewise, in view of \eqref{eq:deltaZbound} and \eqref{fbd-x}, we obtain that
\begin{equation}\label{fbd-z}
\| \delta X'\|_{\beta,m,[0,T]} =\| \delta Z\|_{\beta,m,[0,T]}  \lesssim 2^{CN} ( \lVert \xi \rVert_m + {\theta}{K'} ).   
\end{equation}

Finally, for any $s \in [t_{k}, t_{k +1} ],t \in [t_{\ell}, t_{\ell +1} ]$ for some $0\le k \le \ell \le N-1,$ 
\begin{equation*}
     R^X_{s,t}  = X_{s,t}- X'_s \eta_{s,t}= R^X_{s,t_{k}} + \sum_{j=0}^{L-1} R^X_{t_{j},t_{j+1}} + R^X_{t_\ell, t} + X'_{s,t_\ell} \delta \eta_{t_\ell, t}
\end{equation*}
and so it follows from \eqref{fbd-z} that
\begin{equation}\label{fbd-rx}
 \Psi_{0,T} =  \|\E. R^X \|_{\balpha+\beta,m,[0,T]} \lesssim 2^{CN} (  \lVert \xi \rVert_m + {\theta}{K'} ).
\end{equation}
    

 
{\bf Step 5.} {\sl Bounding $\lVert X \rVert_{S^m}$ by the Stochastic Sewing Lemma.} Let $\bar J_{s,t}:=Z_s \delta \eta_{s,t} + Z'_s \eta^{(2)}_{s,t}.$ Then  
		\begin{align*}
        \delta  \bar J_{s, u, t} &:=  \bar J_{s,t} -   \bar J_{s, u} -  \bar J_{u, t} 
        = -R^{Z}_{s, u} \delta \eta_{u, t} - \delta Z'_{s, u} \mn_{u, t}. 
    \end{align*}
In view of \eqref{est-z'}, \eqref{est-rz}, \eqref{fbd-x}, \eqref{fbd-z} and \eqref{fbd-rx},
we have 
\begin{equation}
    \| \delta Z'\|_{\beta',m} + \| R^Z\|_{\beta, m} +\|\E. R^Z\|_{\beta+\beta',m}  \lesssim 2^{CN} ( \lVert \xi \rVert_m + {\theta}{K'}) . 
\end{equation}
Thus, for any $0\le s \le u \le t \le T,$
\begin{equation*}
    	\lVert \sup_{\tau \in [u, t]} \delta  \bar J_{s, u, \tau} \rVert_m \lesssim 2^{CN} (\lVert \xi \rVert_m + {\theta}{K'})(t - s)^{ \alpha + \beta}
    \end{equation*}
    \begin{equation*}
    	\lVert \E_s \delta  \bar J_{s, u, t} \rVert_m \lesssim 2^{CN} (\lVert \xi \rVert_m + {\theta}{K'} )(t - s)^{ \alpha + \beta +\beta'}
    \end{equation*}

    Thus, it follows from the Stochastic Sewing Lemma \cite[Theorem 2.8]{friz2021roughito} that 
	\begin{equation*}
		  \lVert \sup_{t \in [0, T]} | \int_0^. (FX+f)_r d \bm \eta_r - Z_0 \delta \eta_{0,t} - Z'_0 \eta_{0,t}^{(2)}  |   \rVert_m \lesssim 2^{CN} ( \lVert \xi \rVert_m + {\theta}{K'}), 
	\end{equation*}
 and thus 
 \begin{equation*}
     \| \int_0^. (FX+f)_r d \bm \eta_r \|_{S^m,[0,T]} \lesssim 2^{CN} ( \lVert \xi \rVert_m + {\theta}{K'}).
 \end{equation*}

On the other hand, the Burkholder-Davis-Gundy inequality yields
	\begin{equation*}
        \lVert \sup_{t \in [0, T]}\int_0^t b_r(X_r) dr   \rVert_m +   \lVert \sup_{t \in [0, T]} \int_0^t \sigma_r(X_r) dW_r   \rVert_m \lesssim  2^{CN} ( \lVert \xi \rVert_m + {\theta}{K'}).
    \end{equation*}
It follows that
    \begin{equation*}
    	 \lVert \sup_{t \in [0, T]} \lvert X_t \rvert  \Vert_m \lesssim  2^{CN} ( \lVert \xi \rVert_m + {\theta}{K'}).
    \end{equation*}
        
\end{proof}

\begin{rem}
The a priori estimate above highlights a key distinction from \cite{friz2021roughito}. In that work, the authors require the stronger mixed norm
\[
\big\| \, \| \cdot \mid \mathcal{F} \|_{m} \, \big\|_{n}
:= \Big\| \mathbb{E}\!\left[ |\cdot|^{m} \mid \mathcal{F} \right]^{1/m} \Big\|_{n},
\]
which is vital to control the nonlinear rough coefficient. In contrast, because the rough coefficient in our setting is affine, the natural $\|\cdot\|_{m}$ norm suffices.

\end{rem}

Armed with the above estimates, we are now ready to establish the well-posedness of the rough SDE \eqref{eq:linear_forward_rsde}.

\begin{thm}\label{thm:linear_sde_well_posed}
Let $\bm \eta \in \sroughnogeo$ and let $(b, \sigma, F, F', f, f')$ satisfy Assumption \ref{ass:m-linear-growth_lip}. Then there exists a unique solution $X $ to the rough SDE \eqref{eq:linear_forward_rsde} with 
$$
    (X,FX+f) \in \md^{\bar \alpha, \beta}_{ \eta} L_{m},  \quad \bar \alpha := \alpha \wedge (\frac12 - \frac{1}{n^\sigma}).
$$
\end{thm}

\begin{proof}

In view of Proposition \ref{priorest-lrsde}, we only need to prove the global well-posedness on the larger space $\md^{\beta, \beta'}_{  \eta} L_{m }.$ Since $\md^{\beta, \beta'}_{  \eta} L_{m } \subset \md^{\bar \beta, {\bar \beta}'}_{  \eta} L_{m }$ for any $\bar \beta \le \beta$ and ${ \bar \beta}' \le \beta'$ we may w.l.o.g.~assume that $\beta'< \beta< \frac12$.  

We establish the well-posedness of \eqref{eq:linear_forward_rsde} using a standard fixed point argument. To this end, we fix some $\tau \in (0,1)$ to be determined later, 
introduce the set
$$
    \mbb_\tau:= \Big\{(Y,Y') \in \md^{\beta, \beta'}_{  \eta} L_{m }([0,\tau];\R^{d_Y})  \Big|\  Y_0=\xi,\ Y'_0=F_0 \xi + f_0,   \Big\}
$$
that we equip with the norm $\|\cdot \|_{ \beta, \beta', m,[0,\tau]} $ and consider the mapping 
$$
\begin{array}{lcccl}
   \Psi: & \mb_\tau  & \rightarrow  & \mb_\tau&\\
    & (y,y') & \mapsto & (Y,Y')& := \Big( \xi+ \int_0^. b_r(y_r)dr + \int_0^. \sigma_r(y_r)dW_r +\int_0^. (F_r y_r+ f_r )d\bm \eta_r, Fy+f \Big).
\end{array}
$$

Our goal is to show that this mapping is a contraction if $\tau$ is small enough. To show that $\Psi$ maps the set $\mbb_\tau$ into itself, let $(y,y') \in \mbb_\tau$ and set $(Y,Y')= \Psi(y,y').$ According to Lemma \ref{productest}, 
$$  
    Y'=Fy+f \in C^{\beta} L_m. 
$$
Moreover, in view of {Step 1} in the proof of Proposition \ref{priorest-lrsde}, 
$$
    \int_0^. b_r(y_r) dr, \int_0^. \sigma_r(y_r)dW_r \in C^{\beta} L_m
$$    
and by Proposition \ref{prop:roughinteg}, 
$$
    (\int_0^. (Fy+f)_r d\BBeta_r, Fy+f ) \in \D_\eta^{\beta, \beta'}L_m.
$$    
 
Next, we show that $\Psi$ is a contraction if $\tau$ is small. For this, let $(y,y'), (\by,\by') \in \mbb_\tau,$ and $(Y,Y')= \Psi(y,y')$ and $(\bY,\bY')= \Psi(\by,\by')$. Moreover, we set $(z,z'):=(Fy+f, (Fy)'+f')$, define $(\bz,\bz')$ similarly and put 
$$
    (\Delta y, \Delta y'):= (y-\by, y'-\by'). 
$$    
Differences $\Delta Y, \Delta Y', \Delta z, \Delta z'$ are defined similarly. Note that $\Delta Y'= F \Delta y.$ It follows by Lemma \ref{productest} that 
\be\label{Deltay'}
\| \delta \Delta Y' \|_{\beta',m } \lesssim K \|  \delta \Delta y \|_{\beta',m} \lesssim K  \|  \delta \Delta y \|_{\beta,m} \tau^{\beta-\beta'},
\ee
where the implicit constant depends only on $T$. 
Since $\Delta z= F \Delta y = \Delta Y'$ this shows that 
\be\label{deltaz}
\|  \delta \Delta z \|_{\beta,m } \lesssim K \|  \delta \Delta y \|_{\beta,m}.
\ee

Since $\Delta z' = (F\Delta y)'$, it follows again by Lemma \ref{productest} that
\be\label{deltaz'}
\| \delta \Delta z' \|_{\beta',m }   \lesssim K ( \|  \delta \Delta y \|_{\beta,m}  +\|  \delta \Delta y' \|_{\beta',m} )
\ee
and because $\Delta z'_0=0,$
\be\label{deltasupz'}
    \sup_{r\in [0,\tau]}\| \Delta  z'_r \|_{ m }   \lesssim K ( \|  \delta \Delta y \|_{\beta,m}  +\|  \delta \Delta y' \|_{\beta',m} ).
\ee

Similarly, since $R^{\Delta z}_{s,t} = \delta (F\Delta y)_{s,t} - (F \Delta y)'_s \delta \eta_{s,t},$ we obtain
\be\label{edeltaz}
\| \E. R^{\Delta z} \|_{\beta+ \beta', m} \lesssim K (\| \E. R^{\Delta y} \|_{\beta+ \beta', m} + \|  \delta \Delta y \|_{\beta,m}). 
\ee

In view of the definition of $\Psi,$ we have for any $s,t \in [0,\tau],$
\[
\delta (\Delta Y)_{s,t} = \int_s^t [b_r(y_r)-b_r(\by_r)] dr +\int_s^t [\sigma_r(y_r)-\sigma_r(\by_r)] dW_r +\int_s^t F_r\Delta y_r  d\bm \eta_r.
\]
By the Burkholder-Davis-Gundy and the H\"older inequality, 
\be\label{Deltabsigma}
\begin{split}
    &  \| \int_s^t [b_r(y_r)-b_r(\by_r)] dr \|_m \lesssim K \|  \delta \Delta y\|_{\beta, m} (t-s), \\
    & \| \int_s^t [\sigma_r(y_r)-\sigma_r(\by_r)] dW_r \|_m \lesssim K \|  \delta \Delta y\|_{\beta, m} (t-s)^{\frac12}.    
\end{split}
\ee

Let us now set $L_t:= \int_0^t F_r\Delta y_r  d\bm \eta_r$. It follows by \cite[Theorem 3.5]{friz2021roughito} that 
\begin{equation*}
\begin{split}
& \| R^L_{s,t}\|_m \lesssim_{ \bm \eta} (\|  \delta \Delta z\|_{\beta,m} + \sup_{r \in [0,\tau]}\|   \Delta z'_r\|_m )(t-s)^{\alpha + \beta } + (\|\E. R^{\Delta z}\|_{\beta+\beta',m} + \|  \delta \Delta z'\|_{\beta',m}) (t-s)^{\alpha+ \beta + \beta'},   \\
&\| \E_s R^L_{s,t}\|_m \lesssim_{\bm \eta} (\| \E. R^{\Delta z} \|_{\beta+\beta',m} + \|  \delta \Delta z'\|_{\beta',m} )(t-s)^{\alpha + \beta + \beta'}+  \sup_{r \in [0,\tau]}\|    \Delta z'_r\|_m  (t-s)^{ 2 \alpha}.
\end{split}
\end{equation*}

Substituting \eqref{deltasupz'}, \eqref{deltaz}, \eqref{deltaz'}, \eqref{edeltaz} into the above inequalities yields  
\be\label{est-deltaL}
\begin{split}
    & \| R^L \|_{\beta+ \beta',m} \lesssim_{\BBeta} K (\|  \delta \Delta y\|_{\beta,m} + \|  \delta \Delta y' \|_{\beta',m} + \| \E. R^{\Delta y} \|_{\beta+\beta',m}) \tau^{\alpha-\beta'}  \\
    & \|\E. R^L  \|_{\beta+ \beta',m} \lesssim_{\BBeta} K (\|  \delta \Delta y\|_{\beta,m} + \|  \delta \Delta y' \|_{\beta',m} + \| \E. R^{\Delta y} \|_{\beta+\beta',m}) \tau^{2\alpha-\beta -\beta'}
\end{split}
\ee
and we conclude from \eqref{Deltabsigma} and \eqref{est-deltaL} that
\be\label{Deltay}
\|  \Delta Y  \|_{\beta,m} \lesssim_{\BBeta} K \|(\Delta y, \Delta y')\|_{\beta, \beta', m} \tau^{\frac12 - \beta}.
\ee

Since 
$$
 R^{\Delta Y}_{s,t} = \delta (\Delta Y)_{s,t} - \Delta Y'_s \delta \eta_{s,t} = \int_s^t [b_r(y_r)-b_r(\by_r)] dr +\int_s^t [\sigma_r(y_r)-\sigma_r(\by_r)] dW_r + R^L_{s,t}
$$
it follows by \eqref{Deltabsigma} and \eqref{est-deltaL} that
\be\label{RDeltaY}
\| \E. R^{\Delta Y} \|_{\beta+\beta',m}  \lesssim_{\BBeta} K \| ( \Delta y, \Delta y') \|_{\beta, \beta',m} \tau^{ 2 \alpha- \beta-\beta' }. 
\ee
Moreover, in view of \eqref{Deltay}, \eqref{Deltay'} and \eqref{RDeltaY}, 
$$
\|(\Delta Y, \Delta Y')\|_{\beta, \beta',m} \le C_{\BBeta,K} \| ( \Delta y, \Delta y') \|_{\beta, \beta',m} \tau^{(\frac12-\beta) \wedge (\beta-\beta')}.
$$

Choosing $\tau$ small enough so that $C_{\BBeta,K} \tau^{(\frac12-\beta) \wedge (\beta-\beta')}< \frac12,$ we see that $\Psi $ is a contraction. Since $\tau$ is independent of the initial condition, the argument can be iterated to obtain a unique global solution.
\end{proof}

\begin{rem}
The above theorem provides a generic solution theory for rough SDEs with an affine rough driver, together with local Lipschitz estimates with respect to the underlying rough paths (see the following subsection). Compared with earlier approaches to rough SDEs (see e.g. \cite{frizroughpaths2021} and references therein), we believe that our method is more amenable to discretization, incorporates randomness naturally, and can be readily adapted to a variety of applications, including filtering, mean-field control, and differential games.

\end{rem}

Replacing stochastic by deterministic sewing in the above proof, we obtain the well-posedness of linear RDE (i.e. $\sigma =0$). If $F$ is deterministic and $b=0$, a corresponding result can be found in, e.g.~\cite{frizroughpaths2021}. 

In the following we allow $(F,F')$ to be stochastic but restricted to $\mathscr{D}^{\beta, \beta'}_\eta L_\infty$, compared to the rough SDE case where coefficients are assumed to be in $\mathbf{D}^{\beta, \beta'}_\eta L_\infty$.

\begin{cor}\label{coro:linear}
Let $(F,F') \in \D^{\beta, \beta'}_\eta L_\infty $ with $\beta \in  (\frac{1}{4}, \alpha]$, $\beta' \in (0, \beta]$, and $  \alpha  + \beta + \beta' > 1$. Let $b:[0,T] \times \Omega \times \R^{d_X} \rightarrow \R^{d_X}$ be a progressively measurable vector field such that almost surely for any $(t,x, \bar x) \in [0,T] \times \R^{d_X} \times \R^{d_X} ,$
\begin{equation*}
    \begin{split}
       & \lvert b(  t, \ome, x) \rvert \lesssim  1 + \lvert x \rvert \\
      &    \lvert b(  t, \ome, x) - b(  t, \ome, \bar x)\rvert \lesssim  \lvert x - \bar x \rvert.
    \end{split}   
\end{equation*}
Then for any $\xi \in L_\infty,$ there exists a unique solution $X$ with $(X,FX) \in \D^{\alpha, \beta}_\eta L_\infty$ to the RDE
\be\label{eq:rough-purelinear}
X_t = \xi +  \int_0^t b_r ( X_r)  d  r + \int_0^t F_r X_r d \BBeta_r.
\ee
In particular, $X$ is uniformly bounded.
\end{cor}

\subsection{It\^o-Lyons continuity}

The following proposition establishes the continuity of solutions with respect to the parameters of the equation and the driving rough path. This result will later be refined for the case of $\BBeta$-dependent parameters. 

\begin{prop}[It\^o-Lyons continuity]\label{prop:continuity_of_sol_x}
For any $\BBeta, \bBBeta \in \mc^{\alpha},$ let 
\[  
    X:=X^{\xi, \BBeta} \quad \mbox{and} \quad \bar X := X^{\bar \xi, \bar \BBeta} 
\]    
be solutions to the rSDE \eqref{eq:linear_forward_rsde} with respective drivers $\bm \eta$ and $ \bar {\bm \eta}$, initial values $\xi$ and $\bar \xi$ and coefficients 
\[
    (b, \sigma, F, F', f, f') \quad \mbox{and} \quad (\bar b, \bar \sigma, \bar F, \bar F', \bar f, \bar f'). 
\]    
If both equations satisfy Assumption \ref{ass:m-linear-growth_lip} with the same parameters $(\beta,\beta',m,n_b,n_\sigma, K)$, then 
    \begin{equation}
        \begin{split}
        & \lVert X - \bar X \rVert_{S^m} + \lVert \delta(X - \bar X) \rVert_{\bar \alpha; m} + \lVert \E_\bullet [R^{X} - \bar R^{\bar X}] \rVert_{\bar \alpha+\beta  }  \\
        &\quad  \lesssim  \lVert \xi - \bar \xi \rVert_m +  \| b(\bar X) - \bar b (\bar X) \|_{H^{m,n^b}}   + \| \sigma(\bar X) - \bar \sigma (\bar X) \|_{H^{m,n^\sigma}}  \\
            &\quad \quad + \rho_\alpha(\bm \eta, \bm {\bar \eta}) + d_{\eta, \bar \eta, \beta,\beta',\infty} \left((F,F'), ( \bar F, \bar F') \right)+ d_{\eta, \bar \eta, \beta,\beta',m} \left((f,f'), ( \bar f, \bar f')\right)\\
        \end{split}
    \end{equation}
    where $\bar \alpha := \alpha \wedge (\frac12- \frac1{n^\sigma})$. The implied constant depends only on $\alpha, \beta, \beta', M, \bar M$ where
    $M$ is any constant s.t.
    \begin{equation*}
     K+  \vertiii{\BBeta  }_\alpha +    \| K^b \|_{H^{m, n^b}} + \| K^\sigma \|_{H^{m, n^\sigma}} + \|f_0\|_m + \|(f, f')\|_{\beta,\beta', m} 
        \leq M
    \end{equation*}
    and $\bar M$ is a corresponding constant for the parameter  $(\bar b, \bar \sigma, \bar F, \bar F', \bar f, \bar f')$ and initial condition $\bar \xi$.
\end{prop}

\begin{proof}
    The arguments that follow are similar to those given in the proof of Proposition \ref{priorest-lrsde}, {but we apply a Davie-Gr\"onwall-type lemma \cite[Lemma 4.11]{friz2021roughito} instead of the stochastic sewing lemma}. Let 
\[    
    \tilde X := X - \bar X \quad \mbox{and} \quad \Gamma_{s, t} \defeq \sup_{r \in [s, t]}\lVert \tilde X_r \rVert_m, 
\]    
    and set
    \begin{equation}
        \begin{split}
            \theta &\defeq \|b(\bX)-\bar b(\bX)\|_{H^{m,n^b}} + \|\sigma(\bX)-\bar \sigma(\bX)\|_{H^{m,n^\sigma}} + \rho_\alpha(\bm \eta, \bm {\bar \eta})  \\
            &\quad \ + d_{\eta, \bar \eta, \beta,\beta',\infty} \left((F,F'), ( \bar F, \bar F') \right)+ d_{\eta, \bar \eta, \beta,\beta',m} \left((f,f'), ( \bar f, \bar f')\right).
        \end{split}
    \end{equation}
    Furthermore, we put
$$    
    Z \defeq FX + f, \quad Z' \defeq (FX)'+f'= (FF + F')X + Ff + f' \quad \mbox{and} \quad 
    R^Z \defeq \delta Z - Z' \delta \eta; 
$$    
 the processes $\bar Z$, $\bar Z'$, and $R^{\bar Z}$ are defined similarly. Finally, we introduce the ``remainder terms''
    \begin{equation*}
        J_{s, t} = \int_s^t Z_r d\bm \eta_r - \int_s^t \bar Z_r d\bm {\bar \eta}_r - (Z_s \delta \eta_{s, t} + Z'_s \mn_{s, t} ) + (\bar Z_s \delta \bar \eta_{s, t} + \bar Z'_s \mnb_{s,t}).
    \end{equation*}
    In terms of this notation,   
    \be\label{rep-DeltaX}
    \begin{split}
    \delta \tX_{s,t}= &   \int_s^t \left[b_r(X_r) - \bar b_r (\bar X_r) \right]dr + \int_s^t \left[ \sigma_r(X_r) - \bar \sigma_r (\bar X_r) \right] dW_r + J_{s,t} \\
    & \   + (Z_s \delta \eta_{s, t} + Z'_s \mn_{s, t} ) - (\bar Z_s \delta \bar \eta_{s, t} + \bar Z'_s \mnb_{s,t})        
    \end{split}
    \ee
   
    {\bf Step 1.}  We start by estimating the process $\delta \tX_{s,t}$. 
\begin{itemize}
\item The first two terms on the right-hand side in \eqref{rep-DeltaX} can be bounded using the Lipschitz continuity of $b$ and $\sigma$, respectively. It holds that 
    \begin{align*}
        & \left \lVert \int_s^t b_r(X_r) - \bar b_r (\bar X_r) dr \right \rVert_m \\
        & \quad \quad \leq (t - s)^{1 - \frac1{m}}  \left[ \E \int_s^t \lvert b_r(X_r) - \bar b_r (\bar X_r) \rvert^m dr\right]^\frac1{m} \\
        & \quad \quad \leq (t - s)^{1 - \frac1{m}}  \left[ \E \int_s^t \lvert b_r(X_r) - b_r (\bar X_r) \rvert^m dr\right]^\frac1{m} + (t - s)^{1 - \frac1{m}}  \left[ \E \int_s^t \lvert b_r(\bar X_r) - \bar b_r (\bar X_r) \rvert^m dr\right]^\frac1{m} \\
        & \quad \quad \lesssim \Gamma_{s, t} (t - s)  +  \theta (t - s)^{1 - \frac1{n^b}}
    \end{align*}
    where the implicit constant depends only on $T$. Likewise, 
    \begin{align*}
        \left \lVert \int_s^t \left[ \sigma_r(X_r) - \bar \sigma_r (\bar X_r) \right] dW_r \right \rVert_m &\lesssim  \Gamma_{s, t} (t - s)^\frac1{2} + \theta (t - s)^{\frac1{2} - \frac1{n^\sigma}} .
    \end{align*}

    \item The last term on the right-hand side in \eqref{rep-DeltaX} equals
    \begin{align*}
        & (Z_s \delta \eta_{s, t} + Z'_s \mn_{s, t} ) - (\bar Z_s \delta \bar \eta_{s, t} + \bar Z'_s \mnb_{s,t}) \\
        = & Z_s(\delta \eta_{s, t} - \delta \bar \eta_{s, t}) + (Z_s - \bar Z_s) \delta \bar \eta_{s, t} + Z'_s(\mn_{s, t} - \mnb_{s, t}) + (Z'_s - \bar Z'_s) \mnb_{s, t}.
    \end{align*}
    \item To handle the second term on the right hand side of the above identity, we note that
    \begin{align*}
        \lVert Z_s - \bar Z_s \rVert_m &\leq \lVert F_s - \bar F_s \rVert_\infty \lVert X_s \rVert_m + \lVert F_s \rVert_\infty \lVert \tilde X_s \rVert_m + \lVert f_s - \bar f_s \rVert_m \\
        &\lesssim \Gamma_{s, t} + \theta
    \end{align*}
    where the implied constant only depends on $\sup_{s\in[0,T]} \lVert X_s \rVert_m$ and $K$. Similarly, for the forth term,
    \begin{align*}
        \lVert Z'_s - \bar Z'_s \rVert_m &\leq \lVert F'_s - \bar F'_s \rVert_\infty \lVert X_s \rVert_m + \lVert F'_s \rVert_\infty \lVert \tilde X_s \rVert_m + \lVert f'_s - \bar f'_s \rVert_m \\
        &\quad +\lVert F_s - \bar F_s \rVert_\infty \lVert X'_s \rVert_m + \lVert F_s \rVert_\infty \lVert \tilde X'_s \rVert_m \\
        &\leq \lVert F'_s - \bar F'_s \rVert_\infty \lVert X_s \rVert_m + \lVert F'_s \rVert_\infty \lVert \tilde X_s \rVert_m + \lVert f'_s - \bar f'_s \rVert_m +\lVert F_s - \bar F_s \rVert_\infty \lVert F_s X_s + f_s \rVert_m \\
        &\quad + \lVert F_s \rVert_\infty (\lVert F_s - \bar F_s \rVert_\infty \lVert X_s \rVert_m + \lVert F_s \rVert_\infty \lVert \tilde X_s \rVert_m + \lVert f_s - \bar f_s \rVert_m) \\
        &\lesssim \Gamma_{s, t} + \theta
    \end{align*}
    where the implied constant only depends on $\sup_{s\in[0,T]}\lVert X_s \rVert_m$, $K, \|(f, f')\|_{\beta,\beta', m}$ and $\|f_0\|_m$.
    As a result, 
    \begin{equation*}
        \lVert (Z_s \delta \eta_{s, t} + Z'_s \mn_{s, t} ) - (\bar Z_s \delta \bar \eta_{s, t} + \bar Z'_s \mnb_{s,t}) \rVert_m \lesssim (\Gamma_{s, t} + \theta)(t - s)^{\alpha}
    \end{equation*}
    where the implied constant only depends on $\sup_s \lVert X_s \rVert_m$, $K, \|(f, f')\|_{\beta, \beta', m}$, $\|f_0\|_m$  and $\| \bar {\bm \eta} \|_\alpha$. 
\end{itemize}    
    Combining the above inequalities yields,
 \begin{equation}\label{eq:tildeX}
         \lVert \delta \tilde X_{s, t} \rVert_m \lesssim \lVert J_{s, t} \rVert + \Gamma_{s, t}(t - s)^\alpha + \theta (t - s)^{\frac12 - \frac1{n^\sigma}}.
     \end{equation}

{\bf Step 2.} We proceed to $R^X-\bar R^{\bar X}$. We first notice that
\begin{equation*}
\begin{split}
      & R^X_{s,t} - \bar R^{\bar X}_{s,t}  = \tX_{s,t}- Z_s \eta_{s,t} + \bar Z_s \bar \eta_{s,t}\\
      & \quad \quad = \int_s^t \left[b_r(X_r) - \bar b_r (\bar X_r) \right]dr + \int_s^t \left[ \sigma_r(X_r) - \bar \sigma_r (\bar X_r) \right] dW_r + J_{s,t} 
    +  Z'_s \mn_{s, t}  -   \bar Z'_s \mnb_{s,t}
\end{split}
\end{equation*}
Then proceeding similarly as in \eqref{eq:tildeX}, we see that 
\begin{equation}\label{eq:Rtildex}
        \lVert R^X_{s, t} - \bar R^{\bar X}_{s, t} \rVert_m \lesssim \lVert J_{s, t} \rVert_m + \Gamma_{s, t}(t - s)^\alpha + \theta (t - s)^{\frac12 - \frac1{n^\sigma}}
    \end{equation}
    and
    \begin{equation}\label{eq:ERtildeX}
        \lVert \E_s(R^X_{s, t} - \bar R^{\bar X}_{s, t}) \rVert_m \lesssim \lVert J_{s, t} \rVert_m + \Gamma_{s, t}(t - s)^{2\alpha} + \theta (t - s)^{(1-\frac1{n^b}) \wedge 2 \alpha}.
    \end{equation}

{\bf Step 3.} Now we estimate $\|J_{s,t}\|_m$ in terms of $(\Gamma_{s,t},\theta)$. To this end, we note that
    \be\label{deltaJ-decom}
    \begin{split}
     \delta J_{s, u, t} &= J_{s, t} - J_{s, u} - J_{u, t} \\
        &= R^Z_{s, u}\delta( \eta -   \bar \eta)_{u, t} + \delta Z'_{s, u}(\mn - \mnb)_{u, t} + (R^Z - \bar R^{\bar Z})_{s, u} \delta \bar \eta_{u, t} + \delta (  Z' - \bar Z')_{s, u} \mnb_{u, t}.
    \end{split}
    \ee
The four terms on the right-hand side of the above inequality can be estimated as follows.  
\begin{itemize}
    \item Using that
    $$
        R^Z_{s, t} = R^F_{s, t} X_t + F_s R^X_{s, t} + \delta X_{s, t} \delta F_{s, t} + R^f_{s, t},
    $$
    the first term on the right-hand side of \eqref{deltaJ-decom} can be bounded by 
    \be\label{1}
    \|R^Z_{s,u} \delta (\eta-\bar \eta)_{u,t}\|_m \lesssim \theta (t-s)^{\alpha+\beta}, \quad \mbox{and} \quad \|\E_s[R^Z_{s,u}] \delta (\eta-\bar \eta)_{u,t}\|_m \lesssim \theta (t-s)^{\alpha+\beta+\beta'}.
    \ee
\item Using 
    $$
        \delta Z'_{s, t} = F_t (X'_t - X'_s) + (F_t - F_s)X'_s + F'_t(X_t - X_s) + (F'_t - F'_s)X_s + f'_t - f'_s,
    $$
    the second term can be bounded by 
    \be\label{2}
    \| \delta Z'_{s, u}(\mn - \mnb)_{u, t} \|_m \lesssim \theta (t-s)^{\beta' + 2\alpha}.
    \ee
\item     
For the third term, it follows from $R^Z_{s,t}=R^F_{s,t} X_s + F_s R^X_{s,t} + F_{s,t}X_{s,t}+R^f_{s,t}$ that
\begin{equation*}
    \begin{split}
             R^Z_{s, t} - \bar R_{s, t}^{\bar Z} &= (R^F_{s, t} - \bar R^{\bar F}_{s, t}) X_s + \bar R^{\bar F}_{s, t} \tilde X_s + F_s (R^X_{s, t} - \bar R^{\bar X}_{s, t}) + (F_s - \bar F_s) \bar R^{\bar X}_{s, t} \\
        &\quad+ (\delta (F - \bar F)_{s, t}) \delta X_{s, t} + \bar F_{s, t} \delta \tilde X_{s, t}+ R^f_{s, t} - \bar R^{\bar f}_{s, t} 
    \end{split}
 \end{equation*}
and hence from \eqref{eq:Rtildex} and \eqref{eq:ERtildeX} that
\be\label{3}
 \begin{split}
    & \|(R^Z - \bar R^{\bar Z})_{s, u} \delta \bar \eta_{u, t} \|_m \\
    & \ \ \ \lesssim  \|J_{s,u}\|_m (t-s)^\alpha + \|\delta \tX_{s,u}\|_m(t-s)^{\alpha+\beta} + \theta (t-s)^{\alpha+\beta} + \Gamma_{s,t} (t-s)^{\alpha+\beta}, \\
  & \|\E_s(R^Z - \bar R^{\bar Z})_{s, u} \delta \bar \eta_{u, t} \|_m \\
  & \ \ \ \lesssim  \|J_{s,u}\|_m (t-s)^\alpha +  \|\delta \tX_{s,u}\|_m(t-s)^{\alpha+\beta} + \theta (t-s)^{\alpha+\beta+\beta'} + \Gamma_{s,t} (t-s)^{\alpha+\beta+\beta'},
 \end{split}
 \ee
 where the implicit constant depends on $\|(X,X')\|_{\beta,\beta',m}$, $\|(\bar X,\bar X')\|_{\beta,\beta',m}$, $M$ and $\bar M$, and thus in view of Proposition \ref{priorest-lrsde}, on $(M,\bar M).$  
 
 \item For the fourth term, by linearity of operators $(\cdot)'$ and $\delta (\cdot)_{s,t}$, we have 
\begin{align*}
    \delta (Z'-\bar Z')_{s,t} &= \delta ( \tF X)'_{s,t} + \delta (F \tX)'_{s,t} + \delta \tf_{s,t}. \\
    &= \delta(\tF'X)_{s,t} + \delta(\tF X')_{s,t} + \delta(F'\tX)_{s,t} + \delta(F \tX')_{s,t} + \delta \tf'_{s,t} ,
\end{align*}
where we set $\tF:=F-\bar F$ and similar for $\tF', \tf'.$ Since $\tX'=\tF X + F \tX+ \tf$ it follows that
  \be\label{4}
  \begin{split}
       & \lVert \delta Z'_{s, t} - \delta \bar Z'_{s, t} \rVert_m \\
        &\quad  \lesssim \| \delta \tF'_{s, t} \|_\infty \lVert X_t \rVert_{m} + \|\tF'_s\|_\infty \|\delta X_{s,t}\|_m +  
        \| \delta \tF_{s, t} \|_\infty \lVert X'_t \rVert_{m} + \|\tF_s\|_\infty \|\delta X'_{s,t}\|_m   \\
        & \quad \quad + \| \delta F'_{s, t} \|_\infty \lVert \tX_t \rVert_{m}  + \| F'_s\|_\infty \|\delta \tX_{s,t}\|_m +
  \| \delta F_{s, t} \|_\infty \lVert \tX'_t \rVert_{m} + \| F_s\|_\infty \|\delta \tX'_{s,t}\|_m + \| \delta \tf'_{s,t}\|_m \\
        & \quad \lesssim \lVert \delta \tilde X_{s, t} \rVert_m + (\Gamma_{s, t} + \theta)(t - s)^{ \beta'}. 
  \end{split}
  \ee
\end{itemize}   
    
Using \eqref{deltaJ-decom}, \eqref{1}, \eqref{2}, \eqref{3}, \eqref{4} and \eqref{eq:tildeX}, we get that
    \begin{equation*}
        \lVert  \delta J_{s, u, t} \rVert_m \lesssim \lVert J_{s, u} \rVert_m(t - s)^\alpha + (\Gamma_{s, t} + \theta) (t - s)^{\alpha + \beta}
    \end{equation*}
    and that
    \begin{equation*}
        \lVert \E_s \delta J_{s, u ,t} \rVert_m \lesssim \lVert J_{s, u} \rVert_m(t - s)^\alpha + (\Gamma_{s, t} + \theta) (t - s)^{\alpha + \beta + \beta'}.
    \end{equation*}

{On the other hand,  
$$
J_{s,t}= \left[\int_s^t Z_r d\bm \eta_r - (Z_s \delta \eta_{s, t} + Z'_s \mn_{s, t} ) \right]  - \left[ \int_s^t \bar Z_r d\bm {\bar \eta}_r  - (\bar Z_s \delta \bar \eta_{s, t} + \bar Z'_s \mnb_{s,t}) \right].
$$
By the Step $3$ in the proof of Proposition \ref{priorest-lrsde}, we have 
$$
\|J_{s,t}\|_m \lesssim |t-s|^{\alpha+\beta}, \ \ \|\E_sJ_{s,t}\|_m \lesssim |t-s|^{\alpha+\beta+\beta'}.
$$ }
  Then, according to Lemma \cite[Lemma 4.11]{friz2021roughito}, there exists $l > 0$ depending on the implicit constant, such that for any $\lvert t - s \rvert < l$, we have 
    \begin{equation}\label{eq:Jboundtilde}
        \lVert J_{s, t} \rVert_m \lesssim (\Gamma_{s, t} + \theta)(t - s)^{\alpha + \beta}
    \end{equation}

{\bf Step 4.} We proceed by establishing a global bound on $\Gamma_{s,t}$. Combining the above inequalities, we have   
    \begin{equation*}
        \lVert \tilde X_{s, t} \rVert_m \lesssim (\Gamma_{s, t}+\theta) (t - s)^{\frac12 - \frac1{n^\sigma}}
    \end{equation*}
    where the implicit constant depends only on $M$ and $\tilde M$. It follows from the preceding estimate that 
$$ 
\Gamma_{s, t } \le \lVert \tilde X_{s} \rVert_m +  \lVert \tilde X_{s, t} \rVert_m  \le  \lVert \tilde X_{s} \rVert_m + C_{M,\bar M} \theta l^{\frac12 - \frac1{n^\sigma}} + C_{M,\bar M} \Gamma_{s, t} l^{\frac12 - \frac1{n^\sigma}}.
$$    
By choosing $l$ small enough such that $C_{M,\bar M} l^{\frac12 - \frac1{n^\sigma}}< \frac12,$ we have for any $t-s<l$ that
$$
\Gamma_{s,t} \le 2 \lVert \tilde X_{s} \rVert_m + 2 C_{M,\bar M} \theta l^{\frac12 - \frac1{n^\sigma}}.
$$
It follows for any any partition $\mpe = \{0 = t_0 < t_1 < \dots < t_N = T\}$ with $\lvert \mpe \rvert < l$ that
    \begin{equation*}
	\Gamma_{t_i, t_{i+1}} \le 2 \lVert \tilde X_{t_i} \rVert_m + 2 C_{M,\bar M} \theta l^{\frac12 - \frac1{n^\sigma}}.
    \end{equation*}
Since
	$\Gamma_{t_0, t_{1}} \lesssim_{M,\bar M} \lVert \xi - \bxi \lVert_m + \theta l^{\frac12 - \frac1{n^\sigma}} ,$
    we conclude that
    \begin{equation*}\label{eq:Gamma_bound}
        \Gamma_{0, T} \leq \sum_{i} \Gamma_{t_i, t_{i+1}} \lesssim_{M,\bar M} \lVert  \xi -\bxi \rVert_m + \theta.
    \end{equation*}

{\bf Step 5.} The preceding estimate along with \eqref{eq:Jboundtilde} shows that for any $t-s<l$, 
$$
\|J_{s,t}\|_m \lesssim (\lVert  \xi -\bxi \rVert_m + \theta) (t-s)^{\alpha+\beta}.
$$
Combining the previous arguments, \eqref{eq:tildeX} implies 
$$
\| \delta \tX_{s,t}\|_m \lesssim (\lVert  \xi -\bxi \rVert_m + \theta ) (t-s)^{\bar \alpha}.
$$
Via a partition argument, we obtain from the above local estimate that 
$$
\|\delta \bX \|_{\bar \alpha, m, [0,T]} \lesssim \lVert  \xi -\bxi \rVert_m + \theta .
$$
Similarly, in view of \eqref{eq:ERtildeX}, we have 
$$
\|\E. R^X  - \E. \bar R^{\bar X} \|_{\bar \alpha + \beta,m,[0,T]}\lesssim \lVert  \xi -\bxi \rVert_m + \theta .
$$
Finally, we proceed as Step $5$ in the proof of Proposition \ref{priorest-lrsde}, and obtain 
    \begin{equation*}\label{eq:tildeXBound}
        \lVert \sup_{t \in [0, T]} | \tilde X_t | \rVert_m \lesssim_{M,\bar M}   \lVert \tilde \xi \rVert_m + \theta.
    \end{equation*}
\end{proof}

\subsection{Measurability and continuity w.r.t.~parameters}

Let $X^{\BBeta}$ be the solution to the rSDE \eqref{eq:linear_forward_rsde} with $\BBeta$-dependent coefficients. To define anticipative pathwise control problems, we need the measurability of the solution mapping $$X:[0,T]  \times \Omega \times  \mc^\alpha  \rightarrow \R^{d_X}$$ as a parameter-dependent stochastic process. To prove the equivalence of rough and pathwise control problems we need a strong result, namely the continuity of the solution w.r.t.~the rough driver.  



\begin{definition}\label{def:prog-meas} 
Let $(\Omega, \mathcal{F}, (\mathcal{F}_t)_{t \geq 0}, \P)$ be a filtered probability space with the filtration $(\mathcal{F}_t)_{t \geq 0}$ satisfying the usual conditions, and $(U,\MU)$ be a measurable space. A measurable map     
$$
    Y:(([0,T] \times \Omega) \times U , (\mathcal{B} ([0,T])  \otimes \mathcal{F}) \otimes  \mathcal{U} ) \mapsto (\R^d , \mathcal{B} (\mathbb{R}^d)), \quad (t,\omega,u) \mapsto Y^u(t,\omega)
$$
is called \  {\it $\U$-progressively measurable} if for any $t\in[0,T]$, its restriction on $[0,t]\times \Omega \times U$ is measurable w.r.t. $\MB([0,t]) \otimes \MF_t \otimes \mathcal{U}$. Moreover, we say that $Y$ has a {\it $\U$-progressively measurable version} $\tY$ if $\tY$ is $\U$-progressively measurable and for any $u \in U$, the processes $ \tY(\cdot, u)$ and $Y(\cdot, u)$ are indistinguishable.  
\end{definition}

To guarantee the progressive measurability of the solutions to the rSDEs \eqref{eq:linear_forward_rsde} w.r.t.~the rough path we assume that the following conditions are satisfied. 

\begin{Assumption}\label{jointly-meas.} Suppose $(m,\beta, \beta',\xi, \bar{\alpha})$ satisfies Assumption \ref{ass:m-linear-growth_lip}.  
\begin{itemize}
    \item[$(i)$] For any $\bm \eta \in \mc^\alpha,$ $(b(\cdot, \BBeta),\sigma(\cdot, \BBeta),F(\cdot, \BBeta),F'(\cdot, \BBeta),f(\cdot, \BBeta),f'(\cdot, \BBeta) )$ satisfies Assumption \ref{ass:m-linear-growth_lip}.
\item[$(ii)$] $(b,\sigma)$ and $ (F,F',f,f')$ are $\U$-progressively measurable with $(U,\MU)=(\R^{d_X} \times \mc^{\alpha}, \MB(\R^{d_X}) \otimes \MB(\mc^{\alpha}) )$ and $(\mc^{\alpha}, \MB(\mc^{\alpha}))$, respectively.
\end{itemize}
\end{Assumption}

\begin{prop}\label{prop:jt-meas}
    Let Assumption \ref{jointly-meas.} hold and let $X^{\bm \eta}$ be the solution to Equation \eqref{eq:linear_forward_rsde} for $\BBeta \in \mc^\alpha$. Then $X$ admits a $\MB(\mc^\alpha)$-progressively measurable version $\bX $.
\end{prop}
\begin{proof} 
By Theorem \ref{thm:linear_sde_well_posed}, local solutions to our rSDE are constructed by Picard iteration in a space of $\eta$-stochastic controlled rough paths, starting from the process 
$$
(t,\ome) \mapsto (\xi   +(F_0 \xi+f_0)\delta \eta_{0,t},F_0 \xi+f_0 ) =: (X^{ (0)}_t(\BBeta), (X^{ (0)}(\BBeta))'_t )(\ome)=: \mathcal{X}^{ (0)}_t(\BBeta,\ome)  
$$ 
and then inductively define 
$$
\mathcal{X}^{(n+1)}(\BBeta):= (X^{ (n+1)}(\BBeta),\ (X^{ (n+1)}(\BBeta))' ):= (X^{ (n+1)}(\BBeta),\  F(\BBeta)X^{ (n)}(\BBeta) + f (\BBeta) ), 
$$
where
$$
X^{ (n+1)}(\BBeta):=\  \xi + \int_0^t b_r(X^{ (n)}_r(\BBeta) , \BBeta)dr + \int_0^t \sigma_r(X^{ (n)}_r(\BBeta) , \BBeta)dB_r
 +\int_0^t  \left( F_r(\BBeta)X^{ (n)}_r(\BBeta) + f_r (\BBeta)   \right) d\BBeta_r.	
$$

According to Theorem \ref{thm:linear_sde_well_posed}, $ \mathcal{X}^{ (n)}(\BBeta)
$ converges to  $\mathrm{X}^{\BBeta}$, uniformly on $[0,T_0]$ in probability for small $T_0$. Since $T_0$ independent of the initial condition, one can easily repeat the following argument and obtain the measurability on the whole interval $[0,T]$. 

We show by induction that the mapping $(t,\omega, \BBeta) \mapsto   \mathcal{X}^{(n)} (t, \omega; \BBeta)$ has a $\MB(\mc^\alpha)$-progressively measurable version. This is obvious if $n=0$. 
For the induction step, we need to consider a Lebesgue integral $\int b (...) dt$, an It\^o integral $\int \sigma (...) dB$ and a rough path integral $\int (...  ) d \BBeta$. In all three cases, the integrand $ (...)$ is $\MB(\mc^\alpha)$-progressively measurable. The desired measurability of the first two integrals then follows from 
Proposition \ref{SY78-Prop5} in the appendix. It remains to show that the mapping
$$
      (t, \omega, \BBeta) \mapsto \int_0^t \left(  F_r(\BBeta)X^{ (n)}_r + f_r (\BBeta)  \right) d\BBeta_r,
$$ 
has a $\MB(\mc^\alpha)$-progressively measurable version. To this end, we recall that the integral 
is defined as the uniform in time limit in $P$-probability of the Riemann-Stieltjes sums 
$$
I^{\Pi,(\BBeta)}_t:= \sum_{[u,v] \in \Pi, u \le t} \left( F_u(\BBeta)X^{(n)}_u + f_u (\BBeta) \right) \eta_{u,v \wedge t} +  \left( F'_u(\BBeta)X^{(n)}_u + F_u(\BBeta)X'^{(n)}_u + f'_u (\BBeta) \right) \eta^{(2)}_{u,v \wedge t}.
$$
The process $I^{\Pi,(\BBeta)}$ is $\MB(\mc^\alpha)$-progressively measurable. By Lemma \ref{prog-limit} with $(U,\MU )  $ in
the lemma taken as $(\mc^\alpha, \MB(\mc^\alpha) )$, we see that the limit $\int_0^.(...) d\BBeta$ has a $\MB(\mc^\alpha)$-progressively measurable version. 

It follows that $X^{(n)}$ has a $\MB(\mc^\alpha)$-progressively measurable version $\bX^{(n)}$ and $  \bX^{(n)}(\cdot , \BBeta)$ converges to $X^{\BBeta}$ uniformly in time in $\P$-probability. The assertion hence follows from Lemma \ref{prog-limit}.
\end{proof}


{In what follows, we identify the solution to our rSDE with its $\MB(\mc^{\alpha})$-progressively measurable version}. The progressive measurability of the solution to our rough SDE allows us to define pathwise control problems with anticipative controls. Continuity of the solution is required to prove their equivalence to rough control problems. 

\begin{Assumption}\label{ass:jointly-conti.} 
    Suppose $(m,\beta, \beta',\xi, \bar{\alpha})$ satisfies Assumption \ref{ass:m-linear-growth_lip}. Moreover, assume that:
    \begin{itemize}
    \item[$(i)$.] For any $\bm \eta \in \mc^\alpha,$ $(b(\cdot, \BBeta),\sigma(\cdot, \BBeta),F(\cdot, \BBeta),F'(\cdot, \BBeta),f(\cdot, \BBeta),f'(\cdot, \BBeta) )$ satisfies Assumption \ref{ass:m-linear-growth_lip}. Moreover, $(b,\sigma)(\cdot,\mathbf{0} )$ satisfies \eqref{ine:Kg} with some $(K^{b,\mathbf{0}},K^{\sigma,\mathbf{0}}) \in H^{m, n^b} \times H^{m,n^{\sigma}}$.
    
\item[$(ii)$.] There exists a modulus function (i.e. an increasing function which is continuous at zero with initial value zero) $\Psi:[0,\infty) \rightarrow [0,\infty),$ such that for any $(t,\ome,x) \in [0,T]\times \Omega \times \R^{d_X},$ $\varphi \in \{b,\sigma\}$, and any $\BBeta, \bBBeta \in \mc^{\alpha},$
\begin{align}\label{ine:lip-bs}
&|\varphi (t,\ome,x,\BBeta) - \varphi(t,\ome,x,\bBBeta)|   \lesssim \Psi( \rho_{\alpha}(\BBeta, \bBBeta) + |\eta_0- \bar\eta_0| );\\ \label{ine:lip-Ff}
& d_{\eta, \bar \eta, \beta,\beta',\infty} \left((F,F')(\BBeta), (   F,   F')(\bBBeta) \right)+ d_{\eta, \bar \eta, \beta,\beta',m} \left((f,f')(\BBeta) , (  f,   f')(\bBBeta) \right) \lesssim \Psi( \rho_{\alpha}(\BBeta, \bBBeta) + |\eta_0- \bar\eta_0| ).
\end{align}
\end{itemize}

\end{Assumption}

\begin{example}
The controlled rough paths $(F,F'), (f,f')$ introduced in Example \ref{exam:controlledrp} satisfy the above condition. If $G:\R \rightarrow \R$ is a bounded smooth function as in Example \ref{exam:controlledrp} (i), and $(F,F')(t,\BBeta):= (G(\eta_t),D_x G(\eta_t) )$, a standard computation shows that $(F,F')$ satisfies \eqref{ine:lip-Ff}. 
\end{example}


The following corollary is a direct consequence of Proposition \ref{prop:continuity_of_sol_x}.

\begin{cor}\label{coro:ct-cts}
  Suppose that Assumption \ref{ass:jointly-conti.} holds. Let $X^{\BBeta}$ be the solution to \eqref{eq:linear_forward_rsde} driven by $\BBeta \in \mc^{\alpha}$. Then the mapping $X^{\BBeta}: \mc^{\alpha} \rightarrow S^m$ is continuous. In particular, the function 
$$
    J(\BBeta):= \E \left[g(X^{\BBeta}_T)+ \int_0^T h_r(X^{\BBeta}_r) dr \right]
$$ 
is continuous on $\mc^{\alpha}$ if the functions $g:\Omega \times  \R^{d_{X}} \rightarrow \R,$ $h:[0,T] \times \Omega \times  \R^{d_{X}} \rightarrow \R$ are such that $h(\cdot,x) \in H^1([0,T];\R)$, $g(x) \in L^1(\R)$ for any $x\in \R^{d_{X}}$, and $(h,g)(t,\cdot)$ is locally Lipschitz continuous in the sense that there exists a constant $C>0$ such that for any $t \in[0,T]$, $x,y \in  \R^{d_{X}},$
$$
|g(x)-g(y)|+ |h(t,x)-h(t,y)|< C(1+|x|+|y|)|x-y|, \ \ \ a.s.
$$


\end{cor}


\section{SMP for affine rough controlled stochastic systems}\label{sec:doss-sussmann}

In this section, we introduce rough stochastic control problems and establish a Stochastic Maximum Principle (SMP) for rough control problems when the rough driver is affine in the state variable. Using an affine rough \emph{Doss-Sussmann transformation} we show that the rough SMP can be reduced to a SMP for standard stochastic control problems. 

Recall that $\mc^{0,\alpha}_g$ is the geometric rough path space, and let $\MB(\mc^{0,\alpha}_g)$ be its Borel $\sigma$-algebra. Here we always assume $\bm \eta = (\eta, \mn) \in \mc^{0, \alpha}_g$ is a geometric rough path. We also assume that controls take values in a convex subset $V$ of a finite-dimensional space with norm $|\cdot |$. The set of admissible controls is denoted by 
\be\label{def:roughcont}
\MA:= \{u:[0,T] \times \Omega \rightarrow V, \text{ progressively measurable,} \ \sup_{t }\E[ |u_t|^p ] <\infty , \ \forall p \ge 2  \}.
\ee

\begin{rem}
    The integrability of controls  guarantees integrability of the coefficients so Proposition \ref{roughito} is applicable. It is possible to relax the integrability by carefully revisiting the proof of \cite[Proposition 2.15]{BCN24}. 
\end{rem}


For any admissible control $u \in \MA$ and initial condition $x_0 \in \R^{d_X}$ we consider the controlled rough SDE\footnote{We could allow for $\BBeta$-dependent coefficients $(b,\sigma)$. We prefer not to do this to simplify the notation.} 
\begin{equation}\label{eq:dynamics_pmp_rough}
    dX^{\bm \eta,   u}_t = b_t(X^{\bm \eta,   u}_t,  u_t   )dt + \sigma_t(X^{\bm \eta,  u}_t, u_t ) dW_t + (F_t(\BBeta) X^{\bm \eta,   u}_t + f_t(\BBeta)) d\bm{\eta}_t, \quad X_0 = x_0
\end{equation}
and introduce the cost functional 
\begin{equation}\label{eq:objective_pmp_rough}
    J (  u; \BBeta) = \E \left[ \int_0^T h_t(X^{\bm \eta,   u}_t, u_t)dt + g(X^{\bm \eta,   u}_T) \right].
\end{equation}
The corresponding value function is denoted
\begin{equation}\label{eq:vf}
   \MV({\BBeta}):=  \inf_{u \in \MA} J(u; \BBeta).
\end{equation}

In Section \ref{sec:ds-trans} and \ref{sec:rSMP}, $\BBeta$ is again fixed, so we omit it in the coefficients $( F,F', f,f').$ The following is our standing assumption, which we assume to hold through Section \ref{sec:ds-trans} and \ref{sec:rSMP}.

\begin{Assumption}\label{ass:control_assumptions}
Let $(\Omega,\mathcal{F},\mathbb{P})$ be a probability space carrying a Brownian motion $W$. Denote by $(\mathcal{F}_t)_{t\in[0,T]}$ the $\mathbb{P}$-augmentation of the filtration generated by $W$.
Let $\beta \in  (\frac{1}{4}, \alpha]$ and $\beta' \in (0, \beta]$ satisfy $\alpha+\beta> \frac12 $ and $\alpha+\beta +\beta'>1.$ We assume that 
$$
    (F, F') \in \D^{\beta,\beta'}_\eta L_{\infty} \quad \mbox{and}  \quad (f, f') \in \md^{\beta,\beta' }_\eta L_{\infty}.
$$
Moreover, the parameters $(b, \sigma, h, g)$ satisfy the following conditions: 
\begin{itemize}
\item[$(i)$] The function $(b,\sigma): [0, T] \times \Omega \times \R^{d_X} \times V \rightarrow \R^{d_X} \times \R^{d_W \times d_X}  $ is measurable, and 
\begin{itemize}
    \item for any $(t,\omega)$ the function $(b,\sigma )(t,\ome,\cdot)$ is continuously differentiable on $\R^{d_X} \times V$ with uniformly (in $(t,\ome,x,v)$) bounded derivatives; 
    \item for all $(x,a)$ the function $(b,\sigma)(\cdot, x,a)$ is progressively measurable and so are all its derivatives; 
    \item the function $(b,\sigma)(t,\ome, 0,0) $ is uniformly bounded.
\end{itemize}

\item[$(ii)$]  The functions $g:  \Omega \times \R^{d_X}  \rightarrow \R$ and $h:[0, T] \times \Omega \times \R^{d_X} \times V \rightarrow \R$ are measurable and
    \begin{itemize}
    \item for any $(x,u)$ the function $h(\cdot, x,u)$ is progressively measurable and $h(\cdot, 0,0) $ is uniformly bounded; 
    \item for any $(t,\ome)$ the functions $g(\ome, \cdot)$ and $ h(t,\ome,\cdot)$ are continuously differentiable on $\R^{d_X}$ and $\R^{d_X} \times V$, respectively and there exists a constant $C>0,$
$$
\lvert \partial_u h(t, x, u) \rvert + \lvert \partial_x h(t, x, u) \rvert \le C( 1 + \lvert x \rvert + \lvert u \rvert) ,\ \ \lvert \partial_x g(x) \rvert \le C (  1 + \lvert x \rvert), \  \ a.s.
$$
\end{itemize}
    
\end{itemize}
\end{Assumption}

Under the above assumption the conditions of Theorem  \ref{thm:linear_sde_well_posed} hold and thus the state dynamics \eqref{eq:dynamics_pmp_rough} is well-posed and the unique solution satisfies
$$
    (X^{\bm \eta,   u}, F X^{\bm \eta,   u} + f) \in \bigcap_{m=2}^{\infty} \BD_{\eta}^{\alpha,\beta}L_m 
$$    

\begin{example}[Linear Quadratic Control Problem]\label{exam:LQ1}
The benchmark case of a LQ problem corresponds to the parameters
\[    
    b_t(x, u) = \tilde A_t x + \tilde B_t u + \tilde b_t, \quad 
    \sigma_t(x, u) = \tilde C_t x + \tilde D_t u + \tilde \sigma_t
\]
and
\[    
    h_t(x, u) = \frac{1}{2}(x^{\top} \tilde{M}_t x + u^{\top} \tilde{N}_t u), \quad g(x) = x^{\top} \tG x 
\]    
    where $\tilde A$, $\tilde B$, $\tilde C$, $\tilde D$, $\tilde M$, $\tilde N$, $\tilde b$, $\tilde \sigma$ are uniformly bounded and progressively measurable processes, $\tilde M$ and $\tilde N$ positively definite matrices, and $\tG$ is a bounded $\MF_T$-measurable random variable. 
\end{example}


\subsection{The rough Doss-Sussmann transformation}\label{sec:ds-trans}

We are now going to rewrite the rough control problem as a standard control problem using the Doss-Sussmann transformation $x \mapsto \phi_\cdot(x)$ where $\phi$ denotes the solution to the (stochastic) rough ODE  
\begin{equation}\label{eq:linear_state_transformation}
	\phi_t(x) = x + \int_0^t (F_t \phi_t(x) + f_t) d\bm \eta_t,  \ \ x\in \R^d.
\end{equation}

In view of Theorem \ref{thm:linear_sde_well_posed}, the above equation is well posed and the solution mapping is affine in the initial state. 

\begin{lemma}\label{lem:existence_of_transformation}
Under Assumption \ref{ass:control_assumptions} equation \eqref{eq:linear_state_transformation} has a unique solution $\phi(x)$ in $\BD_\eta^{  \alpha, \beta}L_m$ for any $m \in [2,\infty)$ and the random mapping $x \mapsto \phi_t(x)$ is affine and invertible for all $t \in [0, T]$. In particular, there exist 
$$
    (A,A') \in  \D_\eta^{\alpha,\beta}L_\infty , \ (\zeta,\zeta') \in \BD_\eta^{\alpha,\beta}L_m, 
$$    
such that 
$$
    \phi_t(x ) = A_t x + \zeta_t \quad \mbox{and} \quad
    \phi^{-1}_t(x)=A^{-1}_t x - A^{-1}_t \zeta_t, 
$$    
where $(A^{-1}, A^{-1, \prime} )\in \D_\eta^{\alpha,\beta}L_\infty$ and $\P$-a.s.~$A^{-1}A(t,\ome)=I_{d_X \times d_{X} }$ for all $t \in [0,T].$
\end{lemma}

\begin{proof}
To show that the mapping $x \mapsto \phi_t(x)$ is affine, let $A$ and $\zeta$ be the solutions of linear stochastic RDE 
\begin{align}\label{eq:A}
        &dA_t(\cdot) = F_t   A_t(\cdot)   d\bm \eta_t   , \ \ \ \ \ \  A_0 = I_{d_X \times d_X} , \\ \label{eq:B}
        &d\zeta_t = F_t  \zeta_t   d{\bm \eta_t}  + f_t d\bm \eta_t, \ \ \ \  \zeta_0 = 0,
\end{align}    
respectively. 
In view of Corollary \ref{coro:linear} and Theorem \ref{thm:linear_sde_well_posed}, the above equations are well posed with solutions 
$$
    (A,FA) \in   \D_\eta^{\alpha,\beta}L_\infty, \quad \mbox{and} \quad 
    (\zeta, F\zeta+f) \in \BD_\eta^{\alpha,\beta}L_m. 
$$    
We claim (i) that $A_t$ is invertible and the inverse $A^{-1}$ satisfies
\begin{equation}\label{eq:inver-matrix}
    dA^{-1}_t(\cdot) =- A^{-1}_t  F_t (\cdot)  d \bm \eta_t    ,\ \ \  A_0^{-1} = I_{d_X \times d_X}.
\end{equation}
and (ii) that $\phi_t(\xi)=A_t \xi + \zeta_t.$ It is then easy to see that $\phi_t$ admits the inverse 
$$
    \phi^{-1}(x)=A^{-1}_t x - A^{-1}_t \zeta_t.
$$
To prove (i) we apply the rough It\^o's formula \eqref{eq:roughito} to the process $ A^{-1}  A x$ for any $x\in \R^d$. We obtain that  
\begin{equation*}
    d A^{-1}_t A_tx = - A^{-1}_t F_t (A_tx) d\bm \eta_t  + A^{-1}_t F_t A_t (x) d \bm \eta_t  =0,
\end{equation*}
from which we conclude that $A^{-1}_t \cdot A_t \equiv I$. To prove the second claim, we set $\psi_t(x):= A_t x + \zeta_t$. Then, 
	\begin{equation*}
	    d \psi_t(x)= dA_t x + d\zeta_t= F_t \psi(\xi ) d \bm \eta_t +   f_t d  \BBeta_t.
	\end{equation*}
Since the rSDE \eqref{eq:linear_state_transformation} admits a unique solution, we conclude that $\phi=\psi.$
\end{proof}

To rewrite the rough control problem as a standard problem, let $ X^{  u}:= X^{\bm {\eta}, u }$ be the unique solution to the rough SDE \eqref{eq:dynamics_pmp_rough}, and $\tX^{  u}$ be the unique solution to the standard SDE
\begin{equation}\label{eq:transformed_dynamics}
	d\tX_t = \tb_t(\tX_t, u_t)dt + \tilde \sigma_t(\tX_t, u_t)dW_t, \quad \tX_0 = x_0
\end{equation}
with (stochastic) coefficients 
\be\label{eq:trans-coeff}
\tb_t(\tx, u) \defeq A_t^{-1} b_t(\phi_t(\tx), u),\ \ \ \tilde \sigma_t(\tx, u) \defeq A_t^{-1} \sigma_t (\phi_t(\tx), u).
\ee


\begin{lemma}\label{lem:transform}
	Suppose that Assumption \ref{ass:control_assumptions} holds. Let $X^{ u}$ be the solution to \eqref{eq:dynamics_pmp_rough} and let $\tX^{ u}$ be the solution to \eqref{eq:transformed_dynamics}. Then $\P$-a.s.  
 \[
    X^{ u}_t = \phi_t(\tX^{ u}_t) = A_t {\tX}^{ u}_t + \zeta_t, 
\]
    where $\phi$ is the unique solution to equation \eqref{eq:linear_state_transformation} and $A$ and $\zeta$ are as in Lemma \ref{lem:existence_of_transformation}. 
    Equivalently,  
$$    
    \tX^{  u}_t = \phi^{-1}_t(X^{  u}_t).
$$
\end{lemma}

\begin{proof}
In view of Theorem \ref{thm:linear_sde_well_posed}, $(A,FA), ( \tX^u, 0) \in \bigcap_{m \geq 2} \BD^{\alpha, \beta}_{\eta}L_m$ and by the rough It\^o formula \eqref{eq:roughito},
\begin{align*}
	d\phi_t(\tX^{  u}_t) &= d(A_t \tX^{  u}_t + \zeta_t) \\
	&= A_t \tb_t(\tX^{  u}_t, u_t)dt + A_t \tilde \sigma_t(\tX^{  u}_t, u_t)dW_t + (F_t \phi_t(\tX^{  u}_t) + f_t)d\bm \eta_t \\
	&= b_t(\phi_t (\tX^{  u}_t), u_t)dt + \sigma_t(\phi_t (\tX^{  u}_t), u_t)dW_t + (F_t \phi_t(\tX^{  u}_t) + f_t)d\bm \eta_t, \quad t \in [0, T].
\end{align*} 
 Hence, $\phi_t (\tX^{  u}_t)$ solves the equation \eqref{eq:dynamics_pmp_rough} and so the assertions follows by the uniqueness of solutions.
\end{proof}

\begin{rem}
    The Doss-Sussmann transformation has been applied to random RDEs (i.e. RDEs with random coefficients) in earlier works such as \cite{CDFO13}. Our rough Doss-Sussmann transformation is different. First, we allow our coefficient functions to be stochastic, more precisely, stochastically controlled; second, our coefficient function is linear in the state process, and thus we have an explicit formula for this transformation. 
\end{rem}
 

The cost function can also be expressed in terms of the state process $\tX^{  u}$. To this end, we set
\be\label{eq:hg}
    \tih(s, \tx, u) \defeq h(s, \phi_s(\tx), u), \quad \mbox{and} \quad 
    \tg(\tx) \defeq  g(\phi_T(\tx))
\ee
and recall that $X^{  u}_t = \phi_t(\tX^{ u}_t)$. Thus,  
\begin{align*}
	J(  u) &= \E \left[ \int_0^T  h(s, X^{  u}_s, u_s)ds + g(X^{  u}_T) \right] \\
	& = \E \left[ \int_0^T  h(s, \phi_t(\tX^{  u}_s), u_s)ds + g(\phi_T(\tX^{ u}_T)) \right] \\
	&= \E \left[ \int_0^T \tih(s,\tX^{u}_s, u_s)ds + \tg(\tX^{ u}_T) \right] =: \tilde J(u).
\end{align*}

As a result, the rough and the transformed control problems are equivalent in the sense that 
\[
    \tMV : = \inf_{u \in \MA} \tJ(u)=\MV. 
\]


\subsection{The rough SMP}\label{sec:rSMP}

The equivalence of the rough and the transformed control problem allows us to derive the desired rough SMP. The Hamiltonian for the transformed control problem is given by 
$$
\tH(t,\tx, \ty, \tz,u):= \tb_t ( \tx, u )^{\top} \ty + \text{Tr} \left[ \tilde \sigma_t(\tx, u)^{\top} \tz \right] + \tih_t(\tx, u).
$$
The standard SMP 
states that if $(u^*, \tX^*)$ is an optimal state/control process for the transformed, then 
$$
 \partial_u \tH(t,\tX^{*}_t, \tY^*_t, \tZ^*_t, u^*_t ) (v-u^*_t) \ge  0,
$$
where $(\tY^*, \tZ^*)$ solves the BSDE 
\be\label{eq:tBSDE}
d\tY_t = - \partial_{\tx} \tH (t, \tX^*, \tY^*_t, \tZ^*_t , u^*_t )dt + \tZ^*_t dW_t, \  \tY_T =  \partial_{\tx}  \tg(  \tX^{*}_T).
\ee

Furthermore, under suitable convexity conditions on the Hamiltonian and the terminal payoff function, if $u^*$ is an admissible control with corresponding state dynamics $\tX^*$, such that 
\[
    \tH(t,\tX^*, \tY^*, \tZ^*,u^*) = \min_{v \in V}\tH(t,\tX^*, \tY^*, \tZ^*,v)
\]
then the control $u^*$ is optimal. To obtain the corresponding SMP for rough stochastic control problems we fix a rough path $\BBeta$ omitting it in the upper index, denote by $A^{-1}$ the solution to the stochastic RDE \eqref{eq:inver-matrix} and introduce the Hamiltonian 
\begin{align}\label{eq:transf-H}
H(t,  x, \ty, \tz, u) :=   \left( A^{-1}_t b_t(x , u) \right)^{\top} \ty + \text{Tr} \left[    \left( A^{-1}_t \sigma_t( x, u)\right)^{\top} \tz \right] +   h_t( x, u).
\end{align}

For any $u \in \MA,$ let $X^u$ be the unique solution to the rSDE \eqref{eq:dynamics_pmp_rough} and let $(\tY^u, \tZ^u)$ be the solution to the (classical) BSDE 
\begin{equation}\label{eq:bsde_standard}
		d\tY_t = - \left( \partial_x b_t(X^u_t,u_t)^{\top} \tY_t + \partial_x \sigma_t(X^u_t,u_t) \tZ_t + A_t \partial_x h (X^u_t,u_t ) \right)dt + \tZ_t dW_t, \  \tY_T = A_T \partial_x  g(  X^{ u}_T).
	\end{equation}
    
The following rough SMP is an immediate consequence of the standard SMP (see e.g. \cite[Theorem 4.4]{peng93} for the necessary part and \cite[Theorem 6.4.6]{pham2009} for the sufficient part) for the transformed control problem.

\begin{thm}[Rough SMP for affine rough control]\label{thm:rough-PMP1}
    Suppose that Assumption \ref{ass:control_assumptions} holds and let $ H$ be the Hamiltonian introduced in \eqref{eq:transf-H}. 
For any $u \in \MA,$ let $X^u$ be the solution to the rSDE \eqref{eq:dynamics_pmp_rough}. Then the following holds.  
\begin{itemize}
    \item[$(1)$] If $(u^*, X^{*})$ is an optimal pair for the rough stochastic control problem \eqref{eq:vf}, then for any $v \in V, $ 
$$
 \partial_u H(t,X^{*}_t, Y^*_t, Z^*_t, u^*_t ) (v-u^*_t) \ge  0, \text{ a.e. } \text{Leb}(dt) \times \P,
$$
where $(Y^*,Z^*)$ is the solution to the BSDE \eqref{eq:bsde_standard} with $(u,X^u)=(u^*, X^*).$ \\

\item[$(2)$] Given $\hat u \in \MA$ and $\hat X$ is the corresponding solution to \eqref{eq:dynamics_pmp_rough}. 
Assume that
$$
H(t,\hat X_t, \hat Y_t, \hat Z_t, \hat u_t)= \min_{v  \in V} H (t,\hat X_t, \hat Y_t, \hat Z_t, v), \ \ \forall  t \in [0,T], \   \P\text{-}a.s.
$$
where $(\hat Y, \hat Z)$ is the corresponding solution to \eqref{eq:bsde_standard}. Moreover, assume that $\P$-a.s. for any $t\in[0,T],$ mappings $(x,u) \mapsto H(t,x, \hat Y_t, \hat Z_t, u)$ and $x \mapsto g(x)$ are convex. Then $\hat u$ is an optimal control, or equivalently, $(\hat u, \hat X)$ is an optimal pair.

\end{itemize}
\end{thm}

\begin{rem}
    $(1).$ The differentiability of coefficients in the control variable can be relaxed. For the sufficient part of the maximum principle (Theorem \ref{thm:rough-PMP1} (2)) it is not needed. The necessary part can also be obtained using different perturbation arguments using higher order differentiability with respect to the state variable as the classical case. $(2).$ We apply the Doss-Sussmann transformation to circumvent the analysis of a rough version of \eqref{eq:bsde_standard}, known as rough BSDEs (see \cite{DF12} for details).
\end{rem}


\subsection{Measurability of optimal controls}\label{sec:meas-control}

{In Section 5 we show how to solve pathwise anticipative control problems using the rough control problem analyzed in this section both on the level of value functions and, more generally, on the level of optimal controls. 

To solve the pathwise LQ we need the $\mb(\mc^{0, \alpha}_g)$-progressively measurability of the optimal control to the corresponding rough control problem. For this, we continue our analysis of the rough LQ problem.}


\begin{example}[LQ problem continued]\label{exam:LQ-rough}
    Let us return to the rough LQ control problem introduced in Example \ref{exam:LQ1}. To simplify the exposition we focus on the one-dimensional case  with deterministic coefficients. Then,
\be\label{eq:rough-H}
    H(t,x,\ty,\tz,u)= \frac12 \tN u^2 + A^{-1}_t(\tB \ty + \tD \tz ) u + A^{-1}_t (\tA x \ty + \tilde{b} \ty + \tC x \tz + \tilde{\sigma} \tz) + \frac12 \tM x^2,
\ee
where we recall $(A,A^{-1}, \zeta)$ is introduced in Lemma \ref{lem:existence_of_transformation}. Moreover, by \eqref{eq:trans-coeff} and \eqref{eq:hg},
\[
    \begin{split}
 &       \tb_t(\tx,u) =   \tilde A   \tx + A^{-1}_t \tilde B u + 
        A^{-1}_t \tilde A \zeta_t + A^{-1}_t \tilde b \\
  &      \tilde \sigma_t(\tx,u) =  \tilde C  \tx + A^{-1}_t \tilde D u + 
        A^{-1}_t \tilde C \zeta_t + A^{-1}_t \tilde \sigma \\
  &      \tih_t(\tx,u) = \frac12 \tM (A_t \tx + \zeta_t)^2     + \frac12 \tN u^2 ;  \ \ \       \tg(\tx) = \tG (A_T \tx + \zeta_T)^2 
    \end{split}
\]
and the associated BSDE \eqref{eq:tBSDE} can be written as
$$
\tY_t= 2 \tG   (A_T^2 \tX^u_T + A_T \zeta_T ) + \int_t^T \left(\tA \tY_r + \tC \tZ_r +   \tM (A_r^2 \tX^u_r + A_t \zeta_r) \right) dr - \int_t^T \tZ_t dW_t.
$$
By Theorem \ref{thm:rough-PMP1} a candidate optimal control is given by  
\be\label{eq:op-fct}
u^*(t,x,\ty,\tz):= {\arg \min}_{u \in V} H(t,x,\ty,\tz, u) = -\tN^{-1} A^{-1}_t (\tB \ty + \tD \tz)
\ee
and we obtain the following coupled (transformed but classical) FBSDE for the optimal state dynamics:  
\be\label{eq:rfbsde}
\left\{
\begin{array}{l}
      \tX_t =  x_0 + \int_0^t \left(\tA \tX_r - \tB \tN^{-1} (A^{-1}_r)^2 (\tB \tY_r + \tD \tZ_r )  + A^{-1}_r \tilde A \zeta_r + A^{-1}_r \tilde b \right)dr \\
     \ \ \ \ \ \ \ \ \  + \int_0^t \left(\tC \tX_r - \tD \tN^{-1} (A^{-1}_r)^2 (\tB \tY_r + \tD \tZ_r )   +  A^{-1}_r \tilde C \zeta_r + A^{-1}_r \tilde \sigma \right) dW_r \\
     \\
      \tY_t =  2  \tG (A^2_T \tX_T + A_T \zeta_T) + \int_t^T \left(\tA \tY_r + \tC \tZ_r + \tM (A_r^2 \tX_r + A_r \zeta_r) \right) dr - \int_t^T \tZ_r dW_r.
\end{array}
\right.
\ee

If we further assume that $\tA=\tC=\tb=\tsig= 0$, and that $F$, $f$ are smooth deterministic functions of time (thus $F'=f'=0$), then we can apply the standard ansatz $\tY_t=P_t \tX_t+ q_t$ for smooth functions $(P,q)$, and apply It\^o's formula to $P\tX + q$ to obtain that 
\[
d \tY_t= P \hB u^* dt + P \hD u^* dW_t + \dot{P}X dt + \dot{q}_t dt, 
\]
where $\hB:=A^{-1}\tB, \hD:=A^{-1}\tD, \hG:= A_T \tG$ and $\hM:= A \tM$.
It then follows from \eqref{eq:op-fct}, \eqref{eq:rfbsde} that 
\be\label{eq:opti-u}
\tZ=P \hD u, \ \  \ u^*=-(\tN+ \hD^2 P  )^{-1} ( \hB  P \tX+ \hB  q),  \ \  \
P \hB u^* + \dot{P} \tX + \dot{q} = - \hM(A\tX + \zeta).
\ee

In view of the representation of the optimal control $u^*$ the pair $(P,q)$ satisfies the (deterministic) Riccati equation
\be\label{eq:reca-P}
\left\{
\begin{array}{l}
\dot{P}_t =  \hB_t^2(\tN+ \hD^{2}_t P_t  )^{-1}   P^2_t - \hM_t A_t\\
P_T= 2 \hG A_T,
\end{array}
\right. 
\ee
\be\label{eq:reca-q}
\left\{
\begin{array}{l}
\dot{q}_t = P_t \hB^2_t(\tN+ \hD^{2}_t P_t )^{-1}   q_t - \hM_t \zeta_t \\
q_T= 2 \hG \zeta_T.
\end{array}
\right.
\ee

 In addition, let us assume that $\tN >0$ and $\tM, \tG$ are non-negative. Since $(A,A^{-1}, \zeta)$ is $\alpha$-H\"older continuous it then follows from \cite[Theorem 5.3]{T03} (see also \cite{W68,B78,P92}), there exists a unique bounded non-negative solution $P$ to \eqref{eq:reca-P}, and so the linear ODE \eqref{eq:reca-q} is also well-posed. Moreover, since $(A,A^{-1}, \zeta)$ depends continuously on the rough path $\BBeta$ so do the coefficients $\hB, \hD, \hM A, \hG A_T$ as mappings from $ \mc_g^{0, \alpha}$ to $C([0,T],\R)$. It then follows by \cite[Theorem 2.1]{KT03} that the mapping $P=P^{\BBeta}:  \mc_g^{0,\alpha} \rightarrow  C([0,T],\R) $ is continuous, and in particular $P^{\BBeta}_t$ is $\mb( \mc_g^{0,\alpha} )$-measurable. The same applies to the mapping $q=q^{\BBeta}$. Thus the forward part of the FBSDE \eqref{eq:rfbsde} is a linear SDE with $\mb( \mc_g^{0, \alpha} )$-progressively measurable coefficients, which implies $\tX$ is exponentially integrable. Then thanks to Proposition \ref{prop:jt-meas}, $\tX$ is also $\mb( \mc_g^{0,\alpha} )$-progressively measurable. It follows that the control $u^*$ given by \eqref{eq:opti-u} belongs to $\MA$, and moreover, $u^*$ has continuous sample paths. 

\end{example}

In general, it is difficult to prove that an optimal control depends measurably $\BBeta$. However, as shown by the following lemma, for any $\vep >0$ there exists an $\mb(\mc^{0, \alpha}_g)$-progressively measurable $\vep$-optimal control. This will later allow us to solve pathwise control problems on the level of value function. 

Different from the bounded coefficient case (\cite{FLZ24}), we need to apply the  Kuratowski-Ryll-Nardzewski Selection Theorem to obtain the measurable $\vep$-optimal control, and thus we need to restrict our space of admissible controls to be Polish. 

 In the above linear quadratic example, the optimal control takes the feedback form \eqref{eq:opti-u}, and we can thus restrict admissible controls to have continuous sample paths. Moreover, since $\tX$ is exponentially integrable, in that setting we can assume that admissible controls admit uniform in time integrability of any order. This motivates the following more restrictive set of admissible controls.


 \begin{Assumption}\label{ass:vep-op}
 Suppose that Assumption \ref{ass:control_assumptions} holds and that for any $\BBeta, \bBBeta \in \mc_g^{0,\alpha}$, the controlled rough paths $(F,F')(\cdot)$ and $(f,f')(\cdot)$ satisfy the condition \eqref{ine:lip-Ff}. Moreover, restrict the set of admissible controls to be
\be\label{eq:roughstochcontrol}
\MA:=   \{u:[0,T] \times \Omega \rightarrow V, \text{ adapted and continuous with } \E[\sup_{t\in[0,T]} |u_t|^p ] <\infty,   \ \forall p \ge 2  \}.
\ee

 \end{Assumption}

\begin{rem}\label{rem:polish-control}

It is standard to show that $\MA$ above is a Polish space under metric 
    $$
    d_{\MA}(u,v):= \sum_{i=1}^\infty 2^{-i} (\|\sup_{t\in[0,T]} u_t\|_i \wedge 1).
    $$
In particular, if $V$ is bounded, then Lemma~\ref{lem:vep_measurable_minimizer} follows by an argument similar to that of \cite[Lemma~5.7]{FLZ24}, without requiring 
$u \in \MA$ to be continuous. 

\end{rem}

The following lemma provides a sufficient condition for the existence of measurable $\epsilon$-optimal controls. 

\begin{lemma}\label{lem:vep_measurable_minimizer}
Let $(\MX,  \Sigma(\MX))$ be a measurable space and $\MA$ be a Polish space. Suppose 
$$
    J: \MX \times \MA \rightarrow  \R,\quad (x,a) \mapsto J(x,a)
$$
is a Carath\'eodory function, i.e. for any $a\in \MA,$ $J(\cdot, a)$ is measurable and for any $x\in \MX,$ $J(x, \cdot)$ is continuous. 
Then, for every $\vep > 0$ there exists a jointly measurable mapping 
$
    a^{\vep}: \MX   \to \MA
$
 such that for any $x\in \MX,$ the control $a^\vep(x) $ is an $\vep$-optimal control for the control problem $\MV(x):=\inf_{a\in \MA}J(x,a)$, that is 
$$
    \inf_{a \in \MA} J(x, a) + \vep \ge J(x, a^\vep(x)).
$$
In particular, under Assumption \ref{ass:vep-op} there exists a $\mb(\mc^{0, \alpha}_g)$-progressively measurable $\vep$-optimal control for the value function \eqref{eq:vf}.
\end{lemma}

\begin{proof}
Let $\vep>0$. For any $x\in \MX,$ consider the set-valued mapping 
$$
\Gamma^\vep (x):= \{a\in \MA: J(x,a)< \MV(x)+ \vep \}.
$$
According to \cite[Lemma 17.7]{AB99}, $\Gamma^\vep (\cdot)$ is a measurable correspondence, and thus weakly measurable (see \cite[Chapter 16]{AB99} for notations). Then by \cite[Lemma 17.3]{AB99}, its closure:
$$
\bar \Gamma^\vep (x)= \{a\in \MA: J(x,a) \le \MV(x)+ \vep \}
$$
is also weakly measurable. In view of the continuity of $J(x,\cdot)$, we see that $\bar \Gamma^\vep (x)$ is closed-set-valued. Then by the Kuratowski-Ryll-Nardzewski Selection Theorem (see \cite[Theorem 17.13]{AB99}), there exists a measurable mapping $a^\vep : \MX \rightarrow \MA$ such that $a^\vep (x) \in \bar \Gamma^\vep (x)$. 

In particular, if $J$ is given by \eqref{eq:objective_pmp_rough} with $\MX=\mc^{0,\alpha}_g$ and $\MA$ given by \eqref{eq:roughstochcontrol}, in view of Remark \ref{rem:polish-control} we have that $\MA$ is Polish. Note that topology induced by $d_{\MA}$ is stronger than the topology induced by $\|\cdot\|_{H^{m,n}}$ for any $n \ge m.$ By Proposition \ref{prop:continuity_of_sol_x}, we have that $J(\BBeta,\cdot)$ is continuous under $d_{\MA}.$ Then by the first part of this lemma, there exists $a^\vep:\mc^{0,\alpha}_g \rightarrow \MA$ which is measurable. Since $\MA$ is separable, for any $k \ge 1, $ there exists a partition of $\MA$, denoted by $\{A_i^k\}_i$, such that $\cup_{i=1} A_i^k= \MA,$ and 
$$
\text{diam}(A_i^k):= \sup_{a,b \in A_i^k} d_{\MA}(a,b) < \frac1k.
$$
Choose any element $a_i^k$ from $A_i^k$ for any $i\ge 1.$
Since $a^\vep$ is measurable, consider a partition of $\mc^{0,\alpha}_g:$
$B^k_{i}:= (a^\vep )^{-1}(A^k_i),$ and let 
$$
u^\vep_k (\BBeta,t,\ome):= \sum_{i} a^k_i(t,\ome)(t,\ome) 1_{B^k_i}(\BBeta)
.$$
Then it follows by definition that for any $\BBeta \in \mc^{0,\alpha}_g,$
$
d_{\MA}(u^\vep_k(\BBeta,\cdot), a^\vep(\BBeta,\cdot))<\frac1k.
$
Now for any $(\BBeta,t,\ome)$ let 
$$
\underline u^\vep(\BBeta,t,\ome):= \liminf_{k} u^\vep_k(\BBeta,t,\ome).
$$
Then we have for any $\BBeta,$ $\underline u^\vep(\BBeta,\cdot)=a^\vep(\BBeta,\cdot).$ Thus $\underline u^\vep(\BBeta,\cdot) \in \MA$ and is an $\vep$-optimal control.


\end{proof}

\section{SMP for anticipative pathwise stochastic control problems}

We are now ready to introduce a probabilistic framework that allows us to define and analyze pathwise control problems with anticipative controls. 
Specifically, we consider the canonical space $\Omega'$ equipped with the Wiener probability $\mathbb{P}'$, $d_B$-dimensional canonical process $B$ and the augmented filtration $(\mathcal{F}'_t)_t$. 
Fix a probability space $(\Omega'', \MF'',\P'')$ that supports a $d_W$-dimensional Brownian motion $W$ and let $\{\MF_t\}_t$ be the augmented filtration of the filtration generated by $W.$ Consider the following probability space as our model for anticipative pathwise control problems:
\be\label{eq:probab}
( {\Omega}, (\MF_t)_t, \P )  = (\Omega', (\MF'_t)_t,\P') \otimes (\Omega'', (\MF''_t)_t, \P'')
\ee
 
 To lift Brownian sample paths $B(\ome')$ to a geometric rough path we define for each $(s,t)\in \Delta$, the Stratonovich integral 
$$
    \mathbb{B}_{s,t}^{\mathrm{Strato}}:=\int_s^t \delta B_{s,r}\circ dB_r 
$$ 
It is well known (see \cite[Chapter 3]{frizroughpaths2021}) that for all $\omega'$ outside a $\P'$-null set $N'$ of $\Omega'$, the {\sl Brownian rough path}
$$
    \BB^{\mathrm{Strato} }(\omega'):=(B(\omega'),\mathbb{B}^{\mathrm{Strato}}(\omega')) 
$$    
belongs to the space $\mc^{0,\alpha}_g$ for all $\alpha \in (1/3,1/2)$ and that the following {\it lifting map} is $(\MF'_t )-$progressively measurable:
\be\label{ito-lift}
\begin{array}{llll}
\BB:=\BB^{\mathrm{Strato}}: & [0,T]\times \Omega' & \rightarrow & (\mc^{0,\alpha}_g, \mb(\mc^{0,\alpha}_g))	\\
 & (t,\ome') &\mapsto & \BB^{\mathrm{Strato} }(\ome')_{.\wedge t}
\end{array}
\ee


\subsection{Controls}
We denote the set of $\MF^W$-adapted admissible controls by 
$$
 \MA:=   \{u:[0,T] \times \Omega'' \rightarrow V, \text{ $ (\MF''_t)_t$-adapted and continuous with } \E''[\sup_{t\in[0,T]} |u_t|^p ] <\infty,   \ \forall p \ge 2  \}
$$
and the set of admissible anticipative controls by

\be\label{con-stoch1}
 \begin{split}
\bMA:= & \Big\{ \bu:[0,T] \times \Omega \rightarrow V, \text{ $ (\MF'_T \otimes \MF''_t)_t$-adapted and continuous with}, \\
&\  \E''[\sup_{t\in[0,T]} |u_t(\ome')|^p ] <\infty ,\ \P'\text{-}a.s. \ \forall p \ge 2
 \Big\}.    
\end{split} 
\ee
which is a subset of $\tilde{\MA}$ from Section \ref{Sec:introtopathwise}.
To highlight the connection between rough and pathwise control problems, we also consider the set of ``rough controls''
\be\label{con-stoch2}
\begin{split}
\MA^o:= & \Big\{u:[0,T] \times \mc^{0,\alpha}_g \times \Omega'' \rightarrow V, \text{$(\MB(\mc_g^{0,\alpha}) \otimes \MF''_t)_t$-adapted and continuous with  }  \\
& \     \E'' [\sup_{t \in [0,T]} |u_t(\BBeta)|^p]     <\infty , \ \forall \BBeta \in \mc^{0,\alpha}_g , \   \forall p \ge 2  \Big\}.
\end{split} 
\ee

The sets of anticipative and rough controls can be identified. For any anticipative $\bu \in \bMA $ and any rough path $\BBeta=(\eta, \eta^{(2)}) \in \mc^{0,\alpha}_g$, we can define a rough control $u^{\BBeta} \in \MA^o$ via
\be\label{eq:u-ueta}
u^{\BBeta}_t(\ome''):=\left\{
\begin{array}{ll}
\bu_t(\eta, \ome''), & \text{if $\bu_t(\eta, \ome'')$ is well-defined};\\
0, & \text{otherwise}
\end{array}
\right.
\ee
 and for any rough $u\in \MA^o, $ we can define an anticipative control $\bu \in \bMA$ by 
\be\label{eq:ueta-u}
\bu_t(\ome', \ome''):=\left\{
\begin{array}{ll}
u_t(\BB(\ome'), \ome''), & \text{if $u_t(\BB(\ome'), \ome'')$ is well-defined};\\
0, & \text{otherwise}.
\end{array}
\right.
\ee


\subsection{Anticipative pathwise control problem.}

We are now going to introduce our rough path approach to anticipative pathwise control problems and show that rough and pathwise problems are equivalent. More precisely, we will put a meaning to an anticipating system with anticipative control $\bu \in \bMA,$
\be\label{eq:stoch-dsde}
dX_t = b_t(X_t, \bu_t )dt + \sigma_t( X_t, \bu_t  ) dW_t + (F_t  X_t + f_t  )\circ d B_t, \quad X_0 = x_0,
\ee
where $(F,f)$ is a semimartingale such that the Stratonovich integral is well-defined. To this end, we denote for any anticipative control $\bu \in \bMA, $ and geometric rough path $\BBeta \in \mc_g^{0, \alpha},$ by $u^{\BBeta}$ the rough control given by \eqref{eq:u-ueta}, and consider the corresponding controlled rSDE
\be\label{eq:stoch-sde}
dX^{ \BBeta,  u^{\BBeta}}_t = b_t(X^{  \BBeta,  u^{\BBeta}}_t, u^{\BBeta}_t )dt + \sigma_t(X^{ \BBeta,  u^{\BBeta}}_t, u^{\BBeta}_t  ) dW_t + (F_t  X^{ \BBeta,   u^{\BBeta}}_t + f_t  ) d\BBeta, \quad X_0 = x_0.
\ee

In view of Proposition \ref{prop:jt-meas} the following assumption guarantees that the rSDE admits a unique $\MB(\mc^{\alpha}_g)$-progressively measurable solution $X^{ \BBeta,  u^{\BBeta}}$. 

\begin{Assumption}\label{ass:pathwise}

The parameters $(b,\sigma, F,F',f,f',g,h)$ satisfy Assumption \ref{ass:vep-op} with $\Omega= \Omega''$.\footnote{We can also allow coefficients depending on $\Omega'$ but we omit it to be consistent with the last section}

\end{Assumption}



We now define the state process $\bar X^{\bar u}$ of the pathwise control problem as the a.s.~randomized solution of the above rough SDE, namely as 
\be\label{eq:sol-randomization}
\bX^{\bu}_t(\ome',\ome'') := \left\{ \begin{array}{ll} X^{\BBeta, u^{\BBeta}}_t (\ome'') |_{\BBeta=\BB(\ome')} & \mbox{for } \ome' \notin N' \\ 0 & \mbox{else.} \end{array} \right. 
\ee

\begin{rem}
If the coefficients of \eqref{eq:stoch-sde} are non-anticipating, i.e.~if $\bu $ is $(\MF'_t \otimes \MF''_t)$-adapted and $(F,F')=(F,0)$, $(f,f')=(f,0)$ are smooth in time, then it follows from \cite[Proposition 8.2]{FLZ24} that the randomised solution $\bX$ is equivalent to a classical solution driven by $(W,B)$, which implies $\bX \in S^2$. 
\end{rem}

It follows from the measurability of the composition of measurable mappings that $\bX^{\bu} $ is $(\MF'_T \otimes \MF''_t)_t$-progressively measurable. Thus, if the cost variable $g(\bX^{\bu}_T)+ \int_0^T h(s, \bX^{\bu}_s, \bu_s) ds$ is integrable on $(\Omega,\MF,\P)$, it is natural to consider the cost functional 
\be\label{eq:def-bj}
    \bJ(\bu;\ome):= \E \left[ g(\bX^{\bu}_T)+ \int_0^T h(s, \bX^{\bu}_s, \bu_s) ds \big| \MF^B_T \right]
\ee
and the corresponding value function 
\be\label{eq:stoch-cost}
\bMV(\ome):=  \text{essinf}_{\bu \in \bMA} \bJ(\bu;\ome).
\ee

The following is our equivalence result. It shows when rough and pathwise problems share the same value function and optimal controls, respectively.   

\begin{prop}\label{prop:v-condi-ep} 
Suppose Assumption \ref{ass:pathwise} holds. Let $\BBeta \in \mc_g^{0,\alpha}$ be any geometric rough path. For any anticipative control $\bu  \in   \bMA,$ let $u^{\BBeta}$ be the control given by \eqref{eq:u-ueta} and $X^{\BBeta,u^{\BBeta}}$ be the solution to \eqref{eq:stoch-sde}, and let $\bX^{\bu}$ be the randomized solution given by \eqref{eq:sol-randomization}. Then the following holds. 

\begin{enumerate}
\item[$(1)$] If the payoff $g(\bX^{\bu}_T)+ \int_0^T h(s, \bX^{\bu}_s, \bu_s) ds$ is integrable on $(\Omega,\MF,\P)$,
then the anticipative and randomised rough cost function coincide a.s.:
\be \label{equiv-condi}
 \bJ(\bu;\ome) = \mathbb{E}  \Big[ g ( X^{\BBeta,u^{\BBeta}}_T) + \int_0^T h(s, X^{\BBeta,u^{\BBeta}}_s, u^{\BBeta}_s) ds
      \Big] \Big|_{\BBeta = \mathbf{B}(\ome')} = J(u^{\BBeta};\BBeta)|_{\BBeta=\BB(\ome')} \quad \P\mbox{-a.s.}
\ee
Similarly, for any $u \in \MA^o,$ let $\bu$ be given by \eqref{eq:ueta-u}. Then 
$$
 \bJ(\bu;\ome) = J(u({\BBeta}, \cdot);\BBeta)|_{\BBeta=\BB(\ome)} \quad \P\mbox{-a.s.}
$$
Furthermore, if for any $\bu \in \bMA$ the resulting payoff $g(\bX^{\bu}_T)+ \int_0^T h(s, \bX^{\bu}_s, \bu_s) ds$ is integrable, then the randomised rough and the anticipative value function coincide:
\be\label{equiv-vf}
\bMV(\ome)= \MV(\BBeta)|_{\BBeta=\BB(\ome')} := \Big( \inf_{u \in \MA} J(u;\BBeta)  \Big)\Big|_{\BBeta=\BB(\ome')} 
\quad \P\mbox{-a.s.}
\ee

\item[$(2)$] 
{Suppose that for any $\bu \in \bMA,$ the payoff $g(\bX^{\bu}_T)+ \int_0^T h(s, \bX^{\bu}_s, \bu_s) ds$ is integrable and let $u^* \in \ma^o$ be an admissible rough control. If for a.e.~any $\BBeta = \BB(\ome')$ the control $u^{*,\BBeta}(\ome''):= u^*(\BBeta, \ome'')$ is optimal for the rough control problem \eqref{eq:vf}, then $\bu^*$ given by \eqref{eq:ueta-u} is optimal for the pathwise problem \eqref{eq:stoch-cost}.
Conversely, let $(\bu^*, \bX^{*})$ be an optimal pair for the anticipative pathwise control problem \eqref{eq:stoch-cost}. For any $\BBeta=\BB(\ome')$ with $\ome' \notin N',$ let $u^{*,\BBeta}$ be given by \eqref{eq:u-ueta} and let $X^{*,\BBeta}$ be the corresponding solution to \eqref{eq:stoch-sde}. Then $(u^{*,\BBeta}, X^{*,\BBeta})$ is the optimal pair for the rough stochastic control problem \eqref{eq:vf}.

}


\end{enumerate}
	
\end{prop}

\begin{proof} 
\begin{enumerate}
\item By linearity, it is enough to show that a.s.
$$
    \mathbb{E}  \big[g (\bX^{\bu}_T) 
     | \MF^B_T  \big] = \mathbb{E}  \big[ g ( X^{\BBeta,u^{\BBeta}}_T) 
      \big ] \Big|_{\BBeta = \mathbf{B}(\ome')}.
$$     
For any bounded measurable function $\phi$ on $(\Omega', \MF^B_T ),$ let $\xi:=\phi(B).$ Since the random variable $X_T^{\BBeta,u^{\BBeta}}$ is jointly measurable on $\mc_g^{0,\alpha} \times \Omega'',$ it follows from Fubini's lemma and the identity \eqref{eq:sol-randomization} for any $\ome' \in N_1^c$ that
$$
\mathbb{E} \big[g (X_T^{\BBeta,u^{\BBeta}}) \big]  \big |_{\mathbf{\BBeta} = \mathbf{B}
     (\omega')}= \mathbb{E}'' \big[g (X_T^{\BBeta,u^{\BBeta}}) |_{\mathbf{\BBeta} = \mathbf{B}
     (\omega')}  \big]= \mathbb{E}'' \big [g ( \bar{X}^{\bu }_T )  \big]. 
$$ 
As a result,
$$
\E \left[ \xi \, \mathbb{E} [g (X_T^{\BBeta,u^{\BBeta}})]  |_{\mathbf{\BBeta} = \mathbf{B}
     (\omega')}  \right]= \E[\xi g(\bar{X}^{\bu }_T)]
$$ 
and the equality of the cost function follows. Regarding the value function, the a.s.~inequality 
$$
    \bMV(\ome) \ge \MV(\BBeta)|_{\BBeta=\BB(\ome')} 
$$    
follows from the definition of admissible controls and \eqref{equiv-condi}. To establish the opposite inequality, we fix $\vep >0$ and recall from Lemma \ref{lem:vep_measurable_minimizer} that there exists an $\vep$-optimal control $u^{\vep } \in \MA^o$. Letting $\bu^{\vep}$ be the corresponding anticipative control given by \eqref{eq:ueta-u} the desired inequality follows from
$$
    \bMV(\ome)-\vep \leq \bJ(\bu^{\vep};\ome) - \vep = J(u^{\vep}(\cdot,\BBeta);\BBeta)|_{\BBeta=\BB(\ome')}-\vep \le \MV(\BBeta)|_{\BBeta=\BB(\ome')} \quad \P\mbox{-a.s.}
$$

\item Let $(\vep,u^\vep, \bar u^\vep)$ be as in (1) and $\bar u^*$ be an optimal anticipative control. Then
\begin{align*}
J(u^{*,\BBeta}; \BBeta)|_{\BBeta=\BB(\ome)} = \bJ(\bu^*; \ome) & =  \text{essinf}_{\bu \in \bMA} \bJ(\bu; \ome) \\
& \le \bJ(\bu^{\vep}; \ome) = J(u^{\vep}(\cdot, \BBeta); \BBeta)|_{\BBeta=\BB(\ome)} \leq \MV(\BBeta)|_{\BBeta=\BB(\ome')} - \vep 
\quad \P\mbox{-a.s.}
\end{align*}
Hence $u^{*,\BBeta}$ is an optimal control for the rough control problem for $\P'$-a.e.~Brownian rough path $\BBeta=\BB(\ome')$. 
Conversely, if $u^{*,\BBeta}$ is an optimal rough control for any $\BBeta=\BB(\ome')$, then $\bu^*$ given by \eqref{eq:ueta-u} belongs to $\bMA$ and for any $\bu \in \bMA$ 
$$
 \bJ(\bu;\ome) = J(u^{\BBeta};\BBeta)|_{\BBeta=\BB(\ome')} \ge J(u^{*,\BBeta}; \BBeta)|_{\BBeta=\BB(\ome')} = \bJ (\bu^* ; \ome).
$$
\end{enumerate}

\end{proof}


\subsection{SMP for generalized pathwise stochastic control problems} 

The integrability condition on the payoff required in Proposition \ref{prop:v-condi-ep} is guaranteed if the functions $g$ and $h$ are bounded or if Assumption \ref{ass:pathwise} holds and the randomized solution $\bar{X}^{\bar u}$ is integrable. Unfortunately, the latter cannot be inferred from the a priori estimate given in Proposition \ref{priorest-lrsde}. The process $X^{\BBeta}$ belongs to $H^1(\Omega'';[0,T])$ for any $\BBeta$. However, since controls are anticipative, this does not guarantee that the randomized process $X^{\bf B}$ belongs to $H^1(\Omega;[0,T])$. Without integrability of the randomized SDE the conditional expectation \eqref{eq:def-bj} may not be well-defined. In particular, the above proposition may not apply to LQ problems as stated. 

To overcome the integrability problem and to extend our analysis to LQ problems we introduce the {\sl generalized cost function} 
\be\label{eq:def-tj}
\tJ(\bu; \ome):= J(u^{\BBeta};\BBeta)|_{\BBeta=\BB(\ome')},
\ee
for the pathwise stochastic problem, along with the corresponding value function 
\be\label{eq:def-tv}
\tMV(  \ome):=  \text{essinf}_{\bu \in \bMA} \tJ(\bu;\ome).
\ee

In what follows, we call the resulting pathwise control problem a {\sl generalized pathwise stochastic control problem.}

\begin{rem}\label{rem:equiv-J}
\begin{itemize}
\item[(1).] Thanks to the joint measurability established in Proposition \ref{prop:jt-meas}, the function $\tilde J$ is measurable, and thus the above essential infimum is well-defined. 
\item[(2).] The generalized cost function coincides with the original one if the randomized SDE is integrable. Specifically, under Assumption \ref{ass:pathwise}, for any $\bu \in \bMA$ s.t.~$(\bX^{\bu}_T,\bX^{\bu}) \in L^1(\Omega) \times H^{1}(\Omega;[0,T]),$  
$$
\bJ(\bu;\ome)= \tJ(\bu;\ome), \ \   \P\text{-}a.s.
$$
Moreover, Proposition \ref{prop:v-condi-ep} holds with the value function \eqref{eq:stoch-cost} replaced by \eqref{eq:def-tv}. 
\end{itemize}
\end{rem}





To state a SMP for generalized pathwise stochastic control problems we consider the (adjoint) BSDE 
\begin{equation}\label{eq:bsde_pathwise}
\begin{split}
		dY^{\BBeta, u^{\BBeta}}_t & = - \left( \partial_x b_t(X^{\BBeta, u^{\BBeta}}_t,u^{\BBeta}_t) Y^{\BBeta, u^{\BBeta}}_t + \partial_x \sigma_t(X^{\BBeta, u^{\BBeta}}_t,u^{\BBeta}_t) Z^{\BBeta, u^{\BBeta}}_t + A^{\BBeta}_t \partial_x h (X^{\BBeta, u^{\BBeta}}_t,u^{\BBeta}_t ) \right)dt
        + Z^{\BBeta, u^{\BBeta}}_t dW_t, \\
     Y^{\BBeta, u^{\BBeta}}_T & = A^{\BBeta}_T \partial_x  g(  X^{\BBeta, u^{\BBeta}}_T),
\end{split}
\end{equation}
where $A^{\BBeta}$ denotes the solution to the rODE \eqref{eq:A}. 
This is a standard BSDE with $\mb(\mc_g^{0,\alpha})$-progressively measurable coefficients, and thus $(Y^{{\BBeta}, u^{\BBeta}}, Z^{{\BBeta}, u^{\BBeta}} )$ is $\mb(\mc_g^{0,\alpha})$-progressively measurable. As a result, the process
\[
(\bY^{\bu}, \bZ^{\bu}):= (Y^{{\BBeta}, u^{\BBeta}}, Z^{{\BBeta}, u^{\BBeta}} )|_{\BBeta=\BB(\ome')}
\]
is $(\MF'_T \otimes \MF''_t)$-progressively measurable. 

In view of Proposition \ref{prop:v-condi-ep}, Remark \ref{rem:equiv-J} and the rough SMP developed in the last section, we obtain the following SMP for generalized pathwise stochastic control problems. The proof follows from Theorem \ref{thm:rough-PMP1}, Proposition \ref{prop:v-condi-ep}, Remark \ref{rem:equiv-J} along with the fact that the adjoint BSDE is a standard It\^o BSDE. 

\begin{thm}[SMP for generalized pathwise stochastic control]\label{thm:path-PMP1}
    Suppose that Assumption \ref{ass:pathwise} holds and for any $\BBeta \in \mc_g^{0,\alpha},$ the Hamiltonian $H= H^{\BBeta}$\footnote{Note that $A^{-1}$ depends on $\BBeta$} is given by \eqref{eq:transf-H}. 
\begin{enumerate}
 \item[$(1)$]{(Necessary)} Suppose that $(\bu^*, \bX^{*})$ is an optimal pair for the generalized pathwise control problem \eqref{eq:def-tv}. Let $u^{*,\BBeta}$ be given by \eqref{eq:u-ueta} with $\bu= \bu^*$, and let $X^{*,\BBeta}$ be the corresponding solution to \eqref{eq:stoch-sde}.
Then for any $v \in V, $ 
$$
 \partial_u H(t,X^{*, \BBeta}_t, Y^{*,\BBeta}_t, Z^{*,\BBeta}_t, u^{*,\BBeta}_t ) (v-u^{*, \BBeta}_t)|_{\BBeta=\BB(\ome)} \ge  0, \text{ a.e. } \text{Leb}(dt) \times \P,
$$
where $(Y^{*,\BBeta},Z^{*,\BBeta})$ is the solution to the above BSDE with $(u^{\BBeta},X^{\BBeta, u^{\BBeta}})=(u^{*,\BBeta}, X^{*,\BBeta}).$ 

\item[$(2)$]{(Sufficient)} Assume that $\hat u \in \bMA$ and $\hat X$ is the corresponding randomised solution given by \eqref{eq:sol-randomization}. Let $(\hu^{\BBeta}, \hX^{\BBeta})$ be given by \eqref{eq:u-ueta}, \eqref{eq:stoch-sde} and let $( \hY^{\BBeta}, \hZ^{\BBeta})$ be the solution to the BSDE above with $ u^{\BBeta}= \hu^{\BBeta}$.
Assume that for $\P'$-a.e. $\ome'$ and $\BBeta=\BB(\ome'),$
\[
H(t,\hat X^{\BBeta}_t, \hat Y^{\BBeta}_t, \hat Z^{\BBeta}_t, \hat u^{\BBeta}_t)  =   \min_{v  \in V} H (t,\hat X^{\BBeta}_t, \hat Y^{\BBeta}_t, \hat Z^{\BBeta}_t, v) , \ \ \forall t \in [0,T], \   \P''\text{-}a.s.
\]
and that $\P''$-a.s. for any $t\in[0,T],$ mappings $(x,u) \mapsto H(t,x, \hat Y^{\BBeta}_t, \hat Z^{\BBeta}_t, u)$ and $x \mapsto g(x)$ are convex. Then $\hat u$ is an optimal control for the generalized problem \eqref{eq:def-tv}, or equivalently, $(\hat u, \hat X)$ is an optimal pair.
\end{enumerate}
\end{thm}

In view of Example \ref{exam:LQ-rough} the above SMP allows us to solve generalized LQ pathwise stochastic control problems. 

\begin{thm}
  Under assumptions in the LQ Example \ref{exam:LQ-rough}, then we have
  $\tMV(\ome)=\MV(\BBeta)|_{\BBeta=\BB(\ome')},$  $\P$-a.s. Moreover, let $u^{*, \BBeta}$ be given by \eqref{eq:opti-u}, and define
  $\bu^*_t(\ome):=u^{*, \BBeta}_t(\ome'')|_{\BBeta=\BB(\ome')}$. Then $\bu^*$ is the optimal control to the anticipative pathwise LQ control problem \eqref{eq:def-tv}.
 \end{thm}

 \begin{rem}
   Our linear–quadratic model goes beyond earlier works on pathwise stochastic control, such as \cite{FLZ24, buckdahn2007pathwise}. Although the randomization argument in this section is similar to the accompanying work \cite{FLZ24} (except here we work with a progressively measurable version which is more familar to the stochastic control community), in contrast to \cite{FLZ24}, the present article focuses on the characterization of the optimal control, particularly within the framework allowing for unbounded coefficients, which results in a substantially different class of admissible controls. 
     
 \end{rem}

\section{Appendix}

This appendix collects two auxiliary results on progressively measurable stochastic processes. All processes are defined on a filtered probability space $(\Omega,\{\MF_t\}_t, \MF, \P)$ that satisfies the usual conditions and supports a Brownian motion $(B_t)_t$, and $(U,\MU)$ is a measurable space. 

The first proposition is taken from \cite[Theorem 18.25]{Olav21}. 

\begin{prop}\label{SY78-Prop5} 
Let $J :  
[0,T] \times \Omega \times U \to \mathbb{R}$ be a $\U$-progressively measurable function and $X$ be a continuous semimartingale, such that for any $u \in U$ and some $p>0,$
\[ 
\mathbb{E} \left[ \Big( \int_0^T (J^u_s)^2
d \langle X \rangle_s \Big)^p \right] < \infty . 
\]
Then the stochastic integral $\left( \int_0^t J^u_s dX_s \right)_{t }$ has a $\U$-progressively measurable version. 
\end{prop}

The next lemma is adapted from \cite[Theorem 62]{protter05}.

\begin{lemma}\label{prog-limit}
 Suppose that $Z^n:[0,T] \times \Omega \times U \rightarrow \R^d$ has a $\MU$-progressively measurable version in the sense of Definition \ref{def:prog-meas} and that $Z^n(\cdot, u)$ is a c\`adl\`ag process for each $u \in U$ and $n \in \N$. Let $Z:[0,T] \times \Omega \times U \rightarrow \R^d$. If the sequence $\{Z^n(\cdot, u)\}_n$ converges to $Z(\cdot,u)$ uniformly in time in probability for each $u \in U$, then $Z$ has a progressively measurable version that is $\P$-a.s. c\`adl\`ag for each $u \in U.$

     
\end{lemma}
 
\begin{proof}
 
Without loss of generality, we assume that $Z^n$ is $\MU$-progressively measurable; else we use its progressive version. The processes 
$$
\Delta_t^{i,j}(\ome, u):=|Z^i-Z^j|(t,\ome,u)
$$
are $\MU$-progressively measurable. Let $n^u_0:=1, $ and for any $k \ge 1,$ let
$$
n_k^u:= \inf \{ n> (k \vee n_{k-1}^u): \sup_{i,j \ge n} P(\sup_{t\in[0,T]}\Delta^{i,j}_t > 2^{-k} ) \le 2^{-k} \}. 
$$
For any $u$, $\Delta^{i,j}(\cdot, u)$ is c\`adl\`ag and adapted, hence $n^u_k:U \rightarrow \N$ is measurable. Thus, for any fixed $(k,t)$, $\tZ^k_t(\ome,u):= Z^{n^u_k}_t(\ome,u)$ is measurable w.r.t. $\MF_t \otimes \MU,$ and for any $(\ome,u)$, is c\`adl\`ag in time. For any $(t,\ome,u),$ let 
$$
\tZ_t(\ome,u):=\left\{
\begin{array}{ll}
	\lim_k \tZ^k_t(\ome,u)  , \ \ & \text{if the limit exists}, \\[2mm]
	0, \ \ &\text{otherwise.}
\end{array}\right.
$$
Then $\tZ$ is measurable and adapted. For each $u \in U,$ since $\tZ^k(\cdot,u)$ converges uniformly in time to $Z(\cdot,u)$ in probability, we have $\P(\text{ for any $t \in [0,T],$ } \tZ_t(\cdot,u)=Z_t(\cdot,u))=1$, and thus $\tZ$ is a version of $Z$ Moreover, since for each fixed $u,$ $Z^n(\cdot, u)$ is c\`adl\`ag and uniformly in time converges to $Z$, we see that $\tZ$ is $\P$-a.s. c\`adl\`ag, and thus $\MU$-progressively measurable.
\end{proof}

\begin{proof}[Proof of Proposition \ref{prop:roughinteg}]
We only show the case when $p<\infty.$ For any $s<u<t,$ let $A_{s,t}:= Z_s \delta \eta_{s,t} + Z'_s \mn_{s,t}$, and we have 
$\delta A_{s,u,t}=-R^Z_{s,u} \delta \eta_{u,t} - Z'_{s,u} \mn_{u,t}.$ It follows by definition that 
\begin{align*}
\| \E_s[ \delta A_{s,u,t}  ]\|_p & \lesssim (\| \delta \eta\|_{\alpha} \|\E_\bullet R^Z\|_{\beta+\beta',p} + \|\mn\|_{2\alpha} \|\E_\bullet \delta Z'\|_{\beta',p})(t-s)^{\alpha+\beta+\beta'}\\
&=:\Gamma_1 (t-s)^{\alpha+\beta+\beta'}\\
 \|\delta A_{s,u,t} \|_p & \lesssim (\|R^Z\|_{\beta,p} \|\delta \eta\|_\alpha+ \|\delta Z'_{s,u} \|_{ p} \|\mn \|_{2\alpha} ) (t-s)^{\alpha+\beta}\\
& \lesssim [(\|\delta Z\|_{\beta,p}+ \sup_{t}\|Z'_t\|_p \|\delta \eta\|_{\alpha} )\|\delta \eta\|_{\alpha}+  \sup_{t}\|Z'_t\|_p \|\mn \|_{2\alpha} ] (t-s)^{\alpha+\beta }\\
& =: \Gamma_2 (t-s)^{\alpha+\beta }.
\end{align*}
where the implicit constant depends only on $(p,\alpha,\beta,\beta',T,d).$
Then according to the Stochastic Sewing lemma, i.e. \cite[Theorem 2.1]{Le20}, there exists a path $\mathcal{I}(Z,Z')_\cdot \equiv \int_0^\cdot (Z,Z') d\BBeta$
such that for any $s<t,$
\begin{align*}
    &\| \int_s^t Z_r d\BBeta_r - Z_s \delta \eta_{s,t}-Z'_s \mn_{s,t} \|_p \lesssim    \Gamma_2 |t-s|^{\alpha+\beta}    +   \Gamma_1 |t-s|^{\alpha+\beta+\beta'},\\  
   & \lVert \mathbb{E}_{s}  ( \int_{s}^{t} Z_{r} d\BX_r - Z_{s} \delta \eta_{s,t} - Z'_s \mn_{s,t}  ) \rVert_{p} 
    \lesssim   \Gamma_1 |t-s|^{\alpha +   \beta + \beta'},
\end{align*}
which implies our conclusion.

\end{proof}






\bibliographystyle{plain}
\bibliography{bibliography}

\begin{thebibliography}{10}

\bibitem{AB99}
Charalambos~D. Aliprantis and Kim~C. Border.
\newblock {\em Infinite Dimensional Analysis: A Hitchhiker's Guide}.
\newblock Springer-Verlag Berlin Heidelberg, 2nd edition, 1999.

\bibitem{allan2020pathwise}
Andrew~L Allan and Samuel~N Cohen.
\newblock Pathwise stochastic control with applications to robust filtering.
\newblock {\em Ann. Appl. Probab.}, 30(5), 2020.

\bibitem{BBFP}
Peter Bank, Christian Bayer, Peter Friz, and Luca Pelizzari.
\newblock Rough pdes for local stochastic volatility models.
\newblock {\em Mathematical Finance}, to appear, 2025+.

\bibitem{bayer2023stability}
Christian Bayer, Peter~K Friz, and Nikolas Tapia.
\newblock Stability of deep neural networks via discrete rough paths.
\newblock {\em SIAM Journal on Mathematics of Data Science}, 5(1):50--76, 2023.

\bibitem{B78}
Jean-Michel Bismut.
\newblock Controle des systemes lineaires quadratiques : Applications de
  l'integrale stochastique.
\newblock In C.~Dellacherie, P.~A. Meyer, and M.~Weil, editors, {\em
  S{\'e}minaire de Probabilit{\'e}s XII}, pages 180--264, Berlin, Heidelberg,
  1978. Springer Berlin Heidelberg.

\bibitem{buckdahn2007pathwise}
Rainer Buckdahn and Jin Ma.
\newblock Pathwise stochastic control problems and stochastic hjb equations.
\newblock {\em SIAM journal on control and optimization}, 45(6):2224--2256,
  2007.

\bibitem{BCN24}
Fabio Bugini, Michele Coghi, and Torstein Nilssen.
\newblock Malliavin calculus for rough stochastic differential equations.
\newblock {\em arXiv:2402.12056 [math.PR]}, 2024.

\bibitem{DB92}
Gabriel Burstein and Mark~H.A. Davis.
\newblock A deterministic approach to stochastic optimal control with
  application to anticipative control.
\newblock {\em Stochastics}, 40(3+4):203--256, 1992.

\bibitem{CDFO13}
D.~Crisan, J.~Diehl, P.~K. Friz, and H.~Oberhauser.
\newblock {Robust filtering: Correlated noise and multidimensional
  observation}.
\newblock {\em The Annals of Applied Probability}, 23(5):2139 -- 2160, 2013.

\bibitem{DK94}
Mark~H.A. Davis and Ioannis Karatzas.
\newblock A deterministic approach to optimal stopping.
\newblock {\em Probability, Statistics and Optimisation ed. FP Kelly)}, pages
  455--466, 1994.

\bibitem{DF12}
Joscha Diehl and Peter Friz.
\newblock {Backward stochastic differential equations with rough drivers}.
\newblock {\em The Annals of Probability}, 40(4):1715 -- 1758, 2012.

\bibitem{Diehl2013StochasticCW}
Joscha Diehl, Peter~K. Friz, and Paul Gassiat.
\newblock Stochastic control with rough paths.
\newblock {\em Applied Mathematics \& Optimization}, 75:285--315, 2013.

\bibitem{frizroughpaths2021}
P.~Friz and M.~Hairer.
\newblock {\em {A course on rough paths}}.
\newblock Springer, 2020.

\bibitem{friz2021roughito}
Peter~K. Friz, Antoine Hocquet, and Khoa Lê.
\newblock Rough stochastic differential equations.
\newblock {\em arXiv preprint arXiv:2106.10340v5}, 2024.

\bibitem{FLZ25}
Peter.~K. Friz, Khoa L\^e, and Huilin Zhang.
\newblock Randomisation of rough stochastic differential equations.
\newblock {\em arXiv:2503.06622}, 2025.

\bibitem{FLZ24}
Peter~K. Friz, Khoa Lê, and Huilin Zhang.
\newblock Controlled rough sdes, pathwise stochastic control and dynamic
  programming principles.
\newblock {\em arXiv:2412.05698 [math.PR]}, 2024.

\bibitem{gassiat2024gradient}
Paul Gassiat and Florin Suciu.
\newblock A gradient flow on control space with rough initial condition.
\newblock {\em arXiv preprint arXiv:2407.11817}, 2024.

\bibitem{Gub04}
M~Gubinelli.
\newblock Controlling rough paths.
\newblock {\em Journal of Functional Analysis}, 216(1):86--140, 2004.

\bibitem{Olav21}
Olav Kallenberg.
\newblock {\em Foundations of Modern Probability, 3rd Ed.}
\newblock Probability Theory and Stochastic Modelling. Springer Nature
  Switzerland AG, 2021.

\bibitem{KT03}
Michael Kohlmann and Shanjian Tang.
\newblock Multidimensional backward stochastic riccati equations and
  applications.
\newblock {\em SIAM Journal on Control and Optimization}, 41(6):1696--1721,
  2003.

\bibitem{Le20}
Khoa L\^e.
\newblock A stochastic sewing lemma and applications.
\newblock {\em Electronic Journal of Probability}, 25:1--55, 2020.

\bibitem{LS98}
Pierre-Louis Lions and Panagiotis~E. Souganidis.
\newblock Fully nonlinear stochastic partial differential equations.
\newblock {\em Comptes Rendus de l'Acad\'emie des Sciences-Series
  I-Mathematics}, 326(9):1085--1092, 1998.

\bibitem{LS98b}
Pierre-Louis Lions and Panagiotis~E. Souganidis.
\newblock Fully nonlinear stochastic partial differential equations: Non-smooth
  equation and applications.
\newblock {\em Comptes Rendus de l'Acad\'emie des Sciences-Series
  I-Mathematics}, 327(8):735--741, 1998.

\bibitem{P92}
Shige Peng.
\newblock Stochastic hamilton–jacobi–bellman equations.
\newblock {\em SIAM Journal on Control and Optimization}, 30(2):284--304, 1992.

\bibitem{peng93}
Shige Peng.
\newblock Backward stochastic differential equations and applications to
  optimal control.
\newblock {\em Applied Mathematics and Optimization}, 27(2):125--144, 1993.

\bibitem{pham2009}
Huy{\^e}n Pham.
\newblock {\em Continuous-time stochastic control and optimization with
  financial applications}, volume~61 of {\em Stochastic Modelling and Applied
  Probability}.
\newblock Springer Berlin Heidelberg, 2009.

\bibitem{protter05}
Philip~E. Protter.
\newblock {\em Stochastic integration and differential equations}, volume~21 of
  {\em Stochastic Modelling and Applied Probability}.
\newblock Springer-Verlag, Berlin, 2005.
\newblock Second edition. Version 2.1, Corrected third printing.

\bibitem{T03}
Shanjian Tang.
\newblock General linear quadratic optimal stochastic control problems with
  random coefficients: Linear stochastic hamilton systems and backward
  stochastic riccati equations.
\newblock {\em SIAM Journal on Control and Optimization}, 42(1):53--75, 2003.

\bibitem{W68}
W.~M. Wonham.
\newblock On a matrix riccati equation of stochastic control.
\newblock {\em SIAM Journal on Control}, 6(4):681--697, 1968.

\end{thebibliography}

\end{document}